\documentclass[12pt]{article}
\usepackage{graphicx}
\usepackage{amsmath,amsthm,amssymb,enumerate}
\usepackage{euscript,mathrsfs}
\usepackage{color}
\usepackage{dsfont}
\usepackage[left=2cm,right=2cm,top=3.5cm,bottom=3.5cm]{geometry}
\usepackage{color}
\usepackage[framemethod=tikz]{mdframed}
\allowdisplaybreaks

\usepackage{soul}

\catcode`\@=11 \@addtoreset{equation}{section}

\catcode`\@=12

\newtheorem{Theorem}{Theorem}[section]
\newtheorem{Proposition}[Theorem]{Proposition}
\newtheorem{Lemma}[Theorem]{Lemma}
\newtheorem{Corollary}[Theorem]{Corollary}

\theoremstyle{definition}
\newtheorem{Definition}[Theorem]{Definition}

\newtheorem{Remark}[Theorem]{Remark}

\newcommand{\bTheorem}[1]{
	\begin{Theorem} \label{T#1} }
	\newcommand{\eT}{\end{Theorem}}

\newcommand{\bProposition}[1]{
	\begin{Proposition} \label{P#1}}
	\newcommand{\eP}{\end{Proposition}}

\newcommand{\bLemma}[1]{
	\begin{Lemma} \label{L#1} }
	\newcommand{\eL}{\end{Lemma}}

\newcommand{\bCorollary}[1]{
	\begin{Corollary} \label{C#1} }
	\newcommand{\eC}{\end{Corollary}}

\newcommand{\bRemark}[1]{
	\begin{Remark} \label{R#1} }
	\newcommand{\eR}{\end{Remark}}

\newcommand{\bDefinition}[1]{
	\begin{Definition} \label{D#1} }
	\newcommand{\eD}{\end{Definition}}

\newcommand{\tvB}{\widetilde{\vB}}

\newcommand{\nb}{\| ({\rm data}) \|}

\newcommand{\Del}{\Delta_x}

\newcommand{\bB}{\vc{B}_B}

\newcommand{\vB}{\vc{B}}

\newcommand{\bfomega}{\boldsymbol{\omega}}

\newcommand{\bfphi}{\boldsymbol{\varphi}}

\newcommand{\bFormula}[1]{
	\begin{equation} \label{#1}}
	\newcommand{\eF}{\end{equation}}

\newcommand{\vrn}{\vr_n}
\newcommand{\vun}{\vu_n}
\newcommand{\vtn}{\vt_n}

\newcommand{\Ov}[1]{\overline{#1}}

\newcommand{\Curl}{{\bf curl}_x}
\newcommand{\DC}{C^\infty_c}

\newcommand{\aleq}{\stackrel{<}{\sim}}

\newcommand{\ageq}{\stackrel{>}{\sim}}

\newcommand{\vr}{\varrho}

\newcommand{\tvt}{\tilde \vt}

\newcommand{\vt}{\vartheta}
\newcommand{\vu}{\vc{u}}

\newcommand{\vc}[1]{{\bf #1}}

\newcommand{\Div}{{\rm div}_x}
\newcommand{\Grad}{\nabla_x}

\newcommand{\dx}{\,{\rm d} {x}}

\newcommand{\dt}{\,{\rm d} t }

\newcommand{\intO}[1]{\int_{\Omega} #1 \ \dx}

\newcommand{\D}{{\rm d}}

\newcommand{\vtB}{\vt_B}

\newcommand{\br}{ \nonumber \\ }

\def\softd{{\leavevmode\setbox1=\hbox{d}%
		\hbox to 1.05\wd1{d\kern-0.4ex{\char039}\hss}}}
\definecolor{Cgrey}{rgb}{0.85,0.85,0.85}
\definecolor{Cblue}{rgb}{0.50,0.85,0.85}
\definecolor{Cred}{rgb}{1,0,0}
\definecolor{fancy}{rgb}{0.10,0.85,0.10}

\newcommand\Cbox[2]{%
	\newbox\contentbox%
	\newbox\bkgdbox%
	\setbox\contentbox\hbox to \hsize{%
		\vtop{
			\kern\columnsep
			\hbox to \hsize{%
				\kern\columnsep%
				\advance\hsize by -2\columnsep%
				\setlength{\textwidth}{\hsize}%
				\vbox{
					\parskip=\baselineskip
					\parindent=0bp
					#2
				}%
				\kern\columnsep%
			}%
			\kern\columnsep%
		}%
	}%
	\setbox\bkgdbox\vbox{
		\color{#1}
		\hrule width  \wd\contentbox %
		height \ht\contentbox %
		depth  \dp\contentbox
		\color{black}
	}%
	\wd\bkgdbox=0bp%
	\vbox{\hbox to \hsize{\box\bkgdbox\box\contentbox}}%
	\vskip\baselineskip%
}

\mdfdefinestyle{MyFrame}{%
	linecolor=black,
	outerlinewidth=1pt,
	roundcorner=5pt,
	innertopmargin=\baselineskip,
	innerbottommargin=\baselineskip,
	innerrightmargin=10pt,
	innerleftmargin=10pt,
	backgroundcolor=white!20!white}

\begin{document}


\title{Compressible magnetohydrodynamics as a dissipative system}

\author{Jan B\v rezina \and Eduard Feireisl
		\thanks{The work of E.F. was partially supported by the
			Czech Sciences Foundation (GA\v CR), Grant Agreement
			21--02411S. The Institute of Mathematics of the Academy of Sciences of
			the Czech Republic is supported by RVO:67985840. }}

\date{}

\maketitle

\centerline{Faculty of Arts and Science, Kyushu University}

\centerline{744 Motooka, Nishi-ku, Fukuoka, 819-0395, Japan}

\medskip

\centerline{Institute of Mathematics of the Academy of Sciences of the Czech Republic}

\centerline{\v Zitn\' a 25, CZ-115 67 Praha 1, Czech Republic}

\begin{abstract}
	
We consider the complete system of equations governing the motion of a general 
compressible, viscous, electrically and heat conductive fluid driven by non--conservative boundary 
conditions. We show the existence of a bounded absorbing set in the energy space and asymptotic 
compactness of trajectories. As a corollary, the set of all entire globally bounded solutions is identified 
as a natural attractor. Examples of boundary conditions giving rise to unbounded solutions are also discussed.	
	
\end{abstract}




\section{Introduction}
\label{i}

Dissipative systems are thermodynamically open systems often in the out of equilibrium regime. 
They are examples of the most interesting real world phenomena developing patterns, structures, or behaviours that they did not have when first formed. The dissipative systems exchange energy and matter with their environment, 
while dissipating mechanical and other forms of energy in accordance with the Second law of thermodynamics. 
In continuum fluid mechanics, dissipative systems are modelled by a family of field equations accompanied 
with \emph{inhomogeneous} boundary conditions. 

Many spectacular events occurring in geophysics or astrophysics, among which the dynamics of the solar convection 
zone, are attributed to the action of the magnetic field, see e.g. Thompson 
and Christensen-{D}alsgaard \cite{ThoChDa}. As direct measurements and data collecting  are often hampered by obvious 
technical limitations, reliable mathematical models are indispensable for understanding the observed phenomena.
As pointed out by Gough \cite{Gough}, a correct choice of boundary conditions plays a crucial role.

The state of a viscous, compressible, electrically and heat conducting fluid is characterized 
by its mass density $\vr = \vr(t,x)$, the (absolute) temperature $\vt = \vt(t,x)$, the velocity 
$\vu = \vu(t,x)$ and the magnetic field $\vB = \vB (t,x)$. The time evolution is governed 
by the system of field equations of \emph{compressible magnetohydrodynamics (MHD)}, see e.g. Weiss and Proctor 
\cite{Weiss}:

\begin{mdframed}[style=MyFrame]
	
	{\bf Equation of continuity:}
	\begin{equation} \label{p1}
		\partial_t \vr + \Div (\vr \vu) = 0.
	\end{equation}
	
	\noindent	{\bf Momentum equation:}
	\begin{equation} \label{p2}
		\partial_t (\vr \vu) + \Div (\vr \vu \otimes \vu) + (\bfomega \times \vr \vu)  + \Grad p (\vr, \vt) = \Div \mathbb{S}(\vt, \Grad \vu) + \Curl  \vc{B} \times \vc{B} +   \vr \Grad M.
	\end{equation}

\noindent
{\bf Induction equation:}
\begin{equation} \label{p3} 
	\partial_t \vc{B} + \Curl (\vc{B} \times \vu ) + \Curl (\zeta (\vt) \Curl  \vc{B} ) = 0,\
	\Div \vc{B} = 0.
	\end{equation}

	\noindent {\bf Internal energy balance:}
	\begin{align}
		\partial_t ( \vr e(\vr, \vt) ) + \Div (\vr e(\vr, \vt) \vu) &+ \Div \vc{q}(\vt, \Grad \vt) \br &=
		\mathbb{S}(\vt, \Grad \vu):\Grad \vu + \zeta(\vt) |\Curl \vB|^2  - p(\vr, \vt) \Div \vu.
		\label{p4}
	\end{align}	

\end{mdframed}

The system is written in a general \emph{rotating frame} 
frequently used in astro/geophysics, where $\bfomega \times \vr \vu$ is the Coriolis force 
and the potential $M$, 
\[
M = G + \frac{1}{2}|\bfomega \times x|^2,     
\]
includes the gravitational component $G = G(t,x)$ as well as the centrifugal force. Note that $G$ may depend on the 
time $t$ in the rotating frame if the source of gravitation is located outside the fluid domain.

We consider a \emph{Newtonian fluid}, 
with the viscous stress tensor
\begin{equation} \label{w8}
	\mathbb{S}(\vt, \Grad \vu) = \mu (\vt) \left( \Grad \vu + \Grad^t \vu - \frac{2}{3} \Div \vu \mathbb{I} \right) +
	\eta(\vt) \Div \vu \mathbb{I},
\end{equation}
where the viscosity coefficients $\mu > 0$ and $\eta \geq 0$ are continuously differentiable functions of the temperature. Similarly, the heat flux obeys \emph{Fourier's law},
\begin{equation} \label{w9}
	\vc{q}(\vt, \Grad \vt)= - \kappa (\vt) \Grad \vt,
\end{equation}	
where the heat conductivity coefficient $\kappa > 0$ is a continuously differentiable function of the temperature.

As pointed out in the introduction, the boundary conditions play a crucial role in the 
long time behaviour of the system. We suppose that the fluid occupies a bounded domain $\Omega \subset R^3$ with a smooth boundary, 
\begin{align} 
\partial \Omega = \Gamma^\vu_D \cup \Gamma^\vu_N = \Gamma^\vt_D \cup \Gamma^\vt_N = 
\Gamma^\vB_D \cup \Gamma^\vB_N, \br \Gamma^\vu_D, \Gamma^\vu_N, \Gamma^\vt_D, \Gamma^\vt_N, 
\Gamma^\vB_D , \Gamma^\vB_N \ \mbox{compact}, \br
\Gamma^\vu_D \cap \Gamma^\vu_N = \emptyset,\ \Gamma^\vt_D \cap \Gamma^\vt_N = \emptyset,\ 
\Gamma^\vB_D \cap \Gamma^\vB_N = \emptyset.
\label{p5}
	\end{align}
Accordingly, each $\Gamma^*_D$, $\Gamma^*_N$ is either empty or coincides with a finite union of connected components of $\partial \Omega$.
We impose the following boundary conditions: 

\begin{mdframed}[style=MyFrame]
	
	{\bf Boundary velocity:}
		\begin{align}  
		\vu|_{\Gamma^\vu_D} &= 0,
	\label{p6} \\ 
		\vu \cdot \vc{n}|_{\Gamma^\vu_N} &= 0,\ 
		[\mathbb{S}(\vt, \Grad \vu) \cdot \vc{n}] \times \vc{n} |_{\Gamma^\vu_N} = 0. 
				\label{p7}
		\end{align}
	
\noindent	
	{\bf Boundary temperature/heat flux:}
	\begin{align}
		\vt|_{\Gamma^\vt_D} &= \vtB, \label{p8} \\ 
		\Grad \vc{q}(\vt, \Grad \vt) \cdot \vc{n}|_{\Gamma^\vt_N} &= 0. 
		\label{p9}
		\end{align} 
	
\noindent{\bf Boundary magnetic field:} 
\begin{align} 
\vB \times \vc{n}|_{\Gamma^\vB_D} &= \vc{b}_\tau, \label{p10} \\
\vB \cdot \vc{n}|_{\Gamma^\vB_N} = b_\nu,\ 
[ (\vB \times \vu) + \zeta \Curl \vB ] \times \vc{n}|_{\Gamma^\vB_N} &= 0. 
 	\label{p11}
	\end{align}
	
	\end{mdframed}

Clearly, the boundary conditions prescribed on $\Gamma^{*}_D$ are of Dirichlet type, while those on 
$\Gamma^{*}_N$ of Neumann type. The fluxes on $\Gamma^{*}_N$ are set to be zero for the sake of simplicity.
More general fluxes are shortly discussed in Section \ref{GF}. 
In accordance with \eqref{p6}, \eqref{p7}, the boundary is impermeable, meaning the fluid velocity 
is always tangent to the boundary. In particular, the total mass 
\begin{equation} \label{p12}
	m_0 = \intO{ \vr(t, \cdot) } 
	\end{equation}  
is a constant of motion. Accordingly, there is no exchange of mass with the 
outer world and the system is therefore driven by the imposed boundary temperature and/or magnetic field.

\subsection{Levinson dissipativity}

The compressible MHD system admits a natural energy 
\[
E(\vr, \vt, \vu, \vB) = \underbrace{\frac{1}{2} \vr |\vu|^2}_{\mbox{kinetic energy}} + \underbrace{\vr e(\vr, \vt)}_{\mbox{internal energy}} + \underbrace{ \frac{1}{2} |\vB|^2 }_{\mbox{magnetic energy}}.
\]
Our goal is to show that the system is \mbox{dissipative} in the sense of Levinson in terms of the 
total energy 
\[
\mathcal{E} = \intO{ E(\vr, \vt, \vu, \vB) }.
\]
Specifically, there exists a universal constant $\mathcal{E}_\infty$ such that 
\begin{equation} \label{p13}
	\limsup_{\tau \to \infty} \intO{ E(\vr, \vt, \vu, \vB)(\tau, \cdot) } < \mathcal{E}_\infty
	\end{equation}
for \emph{any} solution $(\vr, \vt, \vu, \vB)$ of the compressible MHD system defined on $(T, \infty)$. We point out that $\mathcal{E}_\infty$ depends only on the ``data'' but it is the same for any global trajectory. In particular, it is independent of the initial energy of the system. 

The data are:

\begin{itemize}
	\item the total mass $m_0$ of the fluid, cf. \eqref{p12}; 
	\item the gravitational potential $G$; 
	\item the rotation vector $\bfomega$; 
	\item the boundary temperature $\vtB$; 
	\item the boundary tangential magnetic field $\vc{b}_\tau$, the boundary normal magnetic field 
	$b_\nu$.
		\end{itemize}
We shall use the symbol $\nb$ to denote the norm of the above data in suitable function spaces specified below.

Besides the total mass $m_0$, there might be other conserved quantities. As shown e.g. by Bauer, Pauly, and 
Schomburg \cite{BaPaSc}, see also Kozono and Yanagisawa \cite{KozYan}, any vector field $\vc{b}$ defined on $\Omega$ admits a decomposition 
\begin{equation} 
	\label{p14}
	\vc{b} = \Grad P + \vc{h} + \Curl \vc{A}, 
	\end{equation}
where 
\begin{equation} \label{p15}
\vc{h} \in \mathcal{H}(\Omega) = \left\{ \vc{h} \in L^2(\Omega; R^3) \ \Big|\ \Curl \vc{h} = \Div \vc{h} = 0, 
\vc{h} \times \vc{n}|_{\Gamma^\vB_D} = 0, \ \vc{h} \cdot \vc{n}|_{\Gamma^\vB_N} = 0    \right\}.
\end{equation}
It follows from the induction equation \eqref{p3} and our choice of the boundary conditions \eqref{p10}, \eqref{p11} that 
\[
\frac{\D }{\dt} \intO{ \vB \cdot \vc{h} } = 0 \ \mbox{for any}\ \vc{h} \in \mathcal{H}(\Omega).
\]
For the sake of simplicity, we shall assume 
\begin{equation} \label{p16}
\intO{ \vB   \cdot \vc{h} } = 0 \ \mbox{for all} \ \vc{h} \in \mathcal{H}(\Omega).	
	\end{equation}
The general case can be handled in a similar way provided the projection of $\vB$ on the space $\mathcal{H}(\Omega)$
is included in the ``data''. Note that the space $\mathcal{H}(\Omega)$ is always finite dimensional, 
see \cite[Section 5.1]{BaPaSc}. Moreover, it is easy to see that $\mathcal{H} = \{ 0 \}$ whenever $\Omega$ is simply connected and $\Gamma^\vB_D$ connected. If $\Gamma^\vB_N = \partial \Omega$ or $\Gamma^\vB_D = \partial \Omega$, 
the dimension of $\mathcal{H}(\Omega)$ coincides  with the first and second Betti number of $\Omega$, respectively, see Kozono and Yanagisawa \cite{KozYanI}.

In view of \eqref{p13}, the boundary conditions must allow the outflow of the thermal energy. 
We suppose
the pressure $p = p(\vr, \vt)$ and the internal energy $e=e(\vr, \vt)$ are interrelated through \emph{Gibbs' equation}
\begin{equation} \label{p17}
	\vt D s = De + p D \left( \frac{1}{\vr} \right),
\end{equation}
where $s = s(\vr, \vt)$ is the entropy. Consequently, the internal energy balance \eqref{p4} may be reformulated
in the form of \emph{entropy balance equation}
\begin{align}
	\partial_t (\vr s(\vr, \vt)) + \Div (\vr s(\vr, \vt) \vu) &+ \Div \left( \frac{\vc{q}(\vt, \Grad \vt)}{\vt} \right) \br &=
	\frac{1}{\vt} \left( \mathbb{S}(\vt, \Grad \vu) : \Grad \vu - \frac{\vc{q}(\vt, \Grad \vt) \cdot \Grad \vt}{\vt} + \zeta(\vt) |\Curl \vB |^2 \right), \label{p18}
\end{align}
see e.g. Weiss and Proctor \cite{Weiss}. The quantity on the right--hand side of \eqref{p18} represents the 
entropy production rate, and, in accordance with the Second law of thermodynamics, it is always non--negative. 
Consequently, all forms of energy are eventually transformed to heat that must be allowed to leave through $\partial \Omega$. Thus, necessarily,
\begin{equation} \label{p19}
\Gamma^\vt_D \ne \emptyset. 
\end{equation}

Next, to avoid development of rapidly rotating fluid in a rotationally symmetric domain, we impose 
\begin{equation} \label{p20}
	\Gamma^\vu_D \ne \emptyset.
	\end{equation}

Finally, we impose a technical condition to control the boundary magnetic oscillations. 
First, we introduce the class of stationary magnetic boundary data. 

\begin{mdframed}[style=MyFrame]
	
	\begin{Definition} \label{Dew1} {\bf [Stationary magnetic field] }
		
		We say that the boundary data $b_\nu$, $\vc{b}_\tau$ are \emph{stationary}, if there exists 
		a continuously differentiable vector field $\vc{B}_B$ such that 
		\[
		\Div \bB = 0,\ \Curl \bB = 0 \ \mbox{in}\ \Omega,\ \bB \cdot \vc{n}|_{\Gamma^\vB_N} = b_\nu,\ 
		\bB \times \vc{n}|_{\Gamma^\vB_D} = \vc{b}_\tau.
		\]
		
		\end{Definition}
	
	\end{mdframed}

Stationarity imposes certain restrictions to be satisfied by $\vc{b}_\tau$,  
\begin{equation} \label{p20a}
{\rm div}_\tau \vc{b}_\tau = 0 \ \mbox{on}\ \Gamma^\vB_D,\ 
\mbox{where}\ {\rm div}_\tau \ \mbox{denotes the tangential divergence,}
\end{equation}
see Auchmuty and Alexander \cite{AuAl}. To establish the existence of a bounded absorbing set, we need an extra hypothesis imposed on the boundary data if the 
magnetic boundary field \emph{is not} stationary, namely
\begin{equation} \label{p21}
\Gamma^\vB_D \subset \Gamma^\vu_D.
\end{equation}
In accordance with Kirchhoff's rule for currents, we have 
\[
{\rm div}_\tau \vc{b}_\tau + \Curl \vc{B} \cdot \vc{n} = 0 \ \mbox{on}\  \Gamma^\vB_D, 
\]
whenever 
\[
\Div \vc{B}  = 0 \ \mbox{in}\ \Omega,\ \vc{B} \times \vc{n}|_{\Gamma^\vB_D} = \vc{b}_\tau, 
\]
see Alexander and Auchmuty \cite{AuAl}.
Thus prescribing a non--stationary vector field $\vc{b}_\tau$ yields an inhomogeneous boundary condition satisfied by 
the normal component of the electric current $\vc{J}$ proportional to $\Curl \vc{B}$, cf. Section \ref{DN}.

\subsection{Absence of static solutions - an example from astrophysics}

Although all global in time solutions eventually enter a bounded absorbing ball, it is worth noting that the system in question may not admit a static stationary solutions, meaning
\[
\vu = 0, \ \vr = \vr(x), \ \vt = \vt(x), \ \vB = \vB(x).
\]
Indeed consider a simple model of a solar convective zone discussed by Tao et al. \cite{Taoetal}. The domain 
boundary is determined by two concentric balls, 
\begin{equation} \label{pp1}
	\Omega  = \left\{ x \in R^3 \ \Big|\ 0 < r_1 < |x| < r_2 \right\}, 
	\end{equation}
with the gravitational potential $G = \Ov{g} \frac{1}{|x|}$. The boundary temperature distribution
is determined by two positive constants $\Theta_{\rm int}$, $\Theta_{\rm ext}$,  
\begin{equation} \label{pp2}
	\vtB = \Theta_{\rm int} \ \mbox{if}\ |x| = r_1,\ \vtB = \Theta_{\rm ext} \ \mbox{if}\ |x| = r_2,\ 
 	0 < \Theta_{\rm ext} < \Theta_{\rm int}. 
	\end{equation}
Finally, the boundary magnetic field is oriented in the normal direction, 
\begin{equation} \label{pp3}
	\Gamma^\vB_N = \emptyset,\ \vc{b}_\tau = 0.
	\end{equation}
The whole system rotates with the rotation axis $\bfomega \ne 0$.

Suppose the system \eqref{p1}--\eqref{p4} endowed with the above boundary conditions admits a static solution, 
meaning
\begin{align}
\Grad p(\vr, \vt) &= \vr \Grad M + \Curl \vB \times \vB, \br
\Curl (\zeta (\vt) \Curl \vB ) &= 0,\ \Div \vB = 0, \br
\Div \vc{q} (\vt, \Grad \vt) &= \zeta(\vt) |\Curl \vB |^2.
	\label{stat}
	\end{align}
The first observation is 
\[
\Curl \vB = 0, 
\]
which follows easily via multiplying the induction equation on $\vc{B}$ and integrating by parts. Consequently, 
the system reduces to 
\begin{align}
\Grad p(\vr, \vt) = \vr \Grad M,\ \Div \vc{q}(\vt, \Grad \vt) = 0.	
	\end{align} 
Applying $\Curl$ to the first equation, we get 
\begin{equation} \label{stat1}
	\Grad \vr \times \Grad M = 0.
	\end{equation}
Similarly, writing 
\[
\Grad p(\vr, \vt) = \partial_\vr p(\vr, \vt) \Grad \vr + \partial_\vt p(\vr, \vt) \Grad \vt = \vr \Grad M
\]
we deduce from \eqref{stat1} $\Grad \vr \times \Grad \vt = 0$ anticipating the condition 
\begin{equation} \label{stat2}
\frac{\partial p(\vr, \vt)}{\partial \vt} \ne 0.
\end{equation}
We conclude 
\begin{equation} \label{stat3}
	\Grad M \times \Grad \vt = 0 .
	\end{equation}

Since $\vc{q}$ is given by Fourier's law $\vc{q} = - \kappa (\vt) \Grad \vt$, 
the unique solution $\vt$ of $\Div \vc{q}(\vr, \Grad \vt)=0$ with the boundary conditions \eqref{pp2} 
admits a non--zero gradient $\Grad \vt$ parallel to 
the gravitational force
\[
\Grad G \times \Grad \vt = 0.
\]
Consequently, condition \eqref{stat3} implies 
\[
\Grad \vt \times \Grad |\bfomega \times x|^2 = 0, 
\]
which is impossible as soon as $\bfomega \ne 0$.

\subsection{Asymptotic compactness}

The existence of global--in--time \emph{classical} solutions for the compressible MHD system in the 
far from equilibrium regime is not known. Here, we adopt the concept of \emph{weak solution} developed in 
\cite{FeGwKwSG}. Unfortunately, the weak solutions may not be uniquely determined by the initial/boundary data. 
Adapting the approach of Sell \cite{SEL}, see also M\' alek and Ne\v cas \cite{MANE}, we consider solutions 
of the compressible MHD system in the space of \emph{trajectories} with a natural group operation of 
forward time shift:
\[
(\vr, \vt, \vu, \vB) (t, \cdot) \mapsto (\vr, \vt, \vu, \vB)(t + \tau),\ \tau \geq 0.
\]

Anticipating the existence of a bounded absorbing set, we consider a sequence of 
global in time solutions $(\vrn, \vtn, \vun, \vB_n)$ along with a sequence of time shifts 
$\tau_n \to \infty$. We say the system is \emph{asymptotically compact} if 
\[
(\vrn, \vtn, \vun, \vB_n)(\cdot + \tau_n) \to 
(\vr, \vt, \vu, \vB) \ \mbox{locally on compact time intervals}\ (-L, L),
\]
where $(\vr, \vu, \vt, \vB)$ is a solution of the same problem with suitable data defined for $t \in (-\infty, \infty)$. Here ``suitable data'' means limits of the corresponding time shifts of the original data. In particular, 
the problem is the same if the data are independent of time. 

Solutions defined for all $t \in (- \infty, \infty)$ are called \emph{entire solutions}. The set of all bounded 
entire solutions,  
\begin{align}
\mathcal{A} = \Big\{ (\vr, \vt, \vu, \vB) \ \Big| & 
(\vr, \vt, \vu, \vB) - \mbox{solution of the compressible (MHD) system for}\ t \in (-\infty, \infty) \br
& \intO{ E(\vr, \vt, \vu, \vB ) } < \mathcal{E}_\infty \Big\}
\nonumber
\end{align}
is shift invariant and captures the essential features of the long time behaviour of the system.

\subsection{Organization of the paper}

The results of the present paper can be seen as a continuation of an ongoing research programme \cite{FanFeiHof},
\cite{FeiPr},  \cite{FeiSwGw} 
extending the classical theory of the dynamics of incompressible viscous fluid flows (see e.g. Constantin et al. 
\cite{CF1}, \cite{CFT}, Foias et al. \cite{FMRT} ) to the compressible case. To the best of our knowledge, the present paper is the first result addressing the problem of long--time behaviour of the compressible MHD system
far from equilibrium. 
There is a large number of literature concerning the incompressible case, see e.g. Duvaut and Lions \cite{DuvLion}, 
Eden and Libin \cite{EdeLib}, Sermange and Temam \cite{SerTem}, Jiang and Jiang \cite{JiangJiang} and the references therein. 

The concept of weak solution and other preliminaries are introduced in Section \ref{w}. The bulk of the paper 
consists of three parts. In Section \ref{AS}, we show the existence of a bounded absorbing set. Section 
\ref{AC} contains the proof of asymptotic compactness. Applications and possible extensions are discussed in Section \ref{AP}.

\section{Main hypothesis, weak solutions}
\label{w}

Before introducing the concept of weak solution, let us state the main hypotheses concerning the structural properties
of the constitutive relations. 

\subsection{Equation of state}

The hypotheses imposed on the form of the equations of state are 
the same as in \cite[Chapter 4]{FeiNovOpen}. They are based on the Second law of thermodynamics enforced 
through Gibbs' relation \eqref{p17} and the \emph{hypothesis of thermodynamics stability} 
\begin{equation} \label{w4}
	\frac{ \partial p (\vr, \vt) }{\partial \vr} > 0,\
	\frac{ \partial e (\vr, \vt) }{\partial \vt} > 0.
\end{equation}
We suppose the pressure obeys the \emph{thermal equation of state} 
\begin{equation} \label{eosT}
	p(\vr, \vt)  = p_M(\vr, \vt) + p_R(\vt),\ p_R(\vt) = \frac{a}{3} \vt^4,\ a > 0,
\end{equation}
where $p_R$ is the radiation pressure. Augmenting the pressure by a radiation component is not 
only technically convenient but also relevant to problems in astrophysics, cf.
Battaner \cite{BATT}. The internal energy satisfies the \emph{caloric equation of state}
\begin{equation} \label{eosC}
	e(\vr, \vt) = e_M (\vr, \vt) + e_R (\vr ,\vt),\ e_R (\vr ,\vt) = \frac{a}{\vr} \vt^4. 
\end{equation}	 
The gas pressure components $p_M$ and $e_M$ satisfy the relation characteristic for monoatomic gases:
\begin{equation} \label{eosM}
	p_M (\vr, \vt) = \frac{2}{3} \vr e_M(\vr, \vt). 
\end{equation} 

Now, it follows from Gibbs' relation \eqref{p17} applied to $p_M$, $e_M$ that they must take a general form
\begin{align}
	p_M(\vr,\vt) = \vt^{\frac{5}{2}} P \left( \frac{\vr}{\vt^{\frac{3}{2}}  } \right),\ e_M(\vr,\vt) =  \frac{3}{2} \frac{\vt^{\frac{5}{2}} }{\vr} P \left( \frac{\vr}{\vt^{\frac{3}{2}}  } \right)
	\label{w1}	
\end{align}
for some the function $P \in C^1[0,\infty)$. In addition, the hypothesis of thermodynamics stability \eqref{w4}
applied to $p_M$, $e_M$ yields
\begin{equation} \label{w2}
	P(0) = 0,\ P'(Z) > 0,\ \frac{ \frac{5}{3} P(Z) - P'(Z) Z }{Z} > 0 \ \mbox{for all}\ Z \geq 0.
\end{equation} 	
In accordance with \eqref{w1}, the entropy takes the form
\begin{equation} \label{w5}
	s(\vr, \vt) = s_M(\vr, \vt) + s_R (\vr, \vt),\ s_M(\vr, \vt) = \mathcal{S} \left( \frac{\vr}{\vt^{\frac{3}{2}} } \right),\ s_R(\vr, \vt) = \frac{4a}{3} \frac{\vt^3}{\vr},
\end{equation}
where
\begin{equation} \label{w6}
	\mathcal{S}'(Z) = -\frac{3}{2} \frac{ \frac{5}{3} P(Z) - P'(Z) Z }{Z^2}.
\end{equation}

Finally, we impose two technical but physically grounded hypotheses
in the degenerate area \[
\frac{\vr}{\vt^{\frac{3}{2}}} = Z >> 1.\] 
First, it follows from 
\eqref{w2} that the function $Z \mapsto P(Z)/ Z^{\frac{5}{3}}$ is decreasing, and we suppose
\begin{equation} \label{w3}
	\lim_{Z \to \infty} \frac{ P(Z) }{Z^{\frac{5}{3}}} = p_\infty > 0.
\end{equation}
Second, by the same token, the function $Z \mapsto \mathcal{S}(Z)$ decreasing, and we suppose 
\begin{equation} \label{w7}
	\lim_{Z \to \infty} \mathcal{S}(Z) = 0.	
\end{equation}	
Hypothesis \eqref{w3} reflects the effect of the electron pressure in the degenerate area, while 
\eqref{w7} is nothing other than the Third law of thermodynamics.

It follows from \eqref{w1}--\eqref{w7} that
\begin{align}\label{ww}
	p(\vr, \vt) &\approx \vr e(\vr, \vt), \br
	\vr^{\frac{5}{3}} + \vt^4 &\aleq p(\vr, \vt) \aleq  	\vr^{\frac{5}{3}} + \vt^4 + 1, \br
	0 \leq \vr s(\vr, \vt) &\aleq \vt^3 + \vr \left( 1 + [\log \vr]^+ + [\log \vt]^+ \right).
\end{align}
Here and hereafter, the symbol $a \aleq b$ means there is a positive constant $c$ such that $a \leq cb$, 
the symbol $a \approx b$ is used to denote $a \aleq b$ and $b \aleq a$.
The reader may consult \cite[Chapter 4]{FeiNovOpen} for the physical background of the above hypotheses as well 
as the proof of \eqref{ww}.

\subsection{Diffusion and the transport coefficients}

The transport coefficients appearing in \eqref{w8}, \eqref{w9} satisfy
\begin{align}
	0 < \underline{\mu} \left(1 + \vt \right) &\leq \mu(\vt) \leq \Ov{\mu} \left( 1 + \vt \right),\
	|\mu'(\vt)| \leq c \ \mbox{for all}\ \vt \geq 0, \br
	0 &\leq  \eta(\vt) \leq \Ov{\eta} \left( 1 + \vt \right),
	\label{w10}
\end{align}
and
\begin{equation} \label{w11}
	0 < \underline{\kappa} \left(1 + \vt^\beta \right) \leq  \kappa(\vt) \leq \Ov{\kappa} \left( 1 + \vt^\beta \right) \ \mbox{for some}\  \beta > 6.
\end{equation}

Finally, we suppose the coefficient of magnetic diffusion $\zeta = \zeta(\vt)$ is a continuously differentiable function of the temperature,
\begin{equation} \label{w12}
	0 < \underline{\zeta}(1 + \vt) \leq \zeta (\vt) \leq \Ov{\zeta}(1 + \vt),\ |\zeta' (\vt) | \leq c
	\ \mbox{for all}\ \vt \geq 0.	
\end{equation}

\subsection{Boundary data}

It is convenient to assume the boundary data are given as restrictions of functions defined on $\Omega$ and for 
any $t \in R$. Specifically, 
\[
\vtB = \tvt|_{\Gamma^\vt_D},\ 
\vc{b}_\tau = \bB \times \vc{n}|_{\Gamma^\vB_D},\ 
b_\nu = \bB \cdot \vc{n} |_{\Gamma^\vB_N}
\]
for suitable extensions $\tvt$, $\bB$.
Accordingly, certain regularity of the boundary data is necessary. 

We suppose the boundary temperature $\vtB = \vtB(t,x)$ belongs to the class
\begin{align}
	\vtB &\in BC(R; C^2(\Gamma^\vt_D)), 
	\ \partial_t \vtB  \in BC (R; C^1(\Gamma^\vt_D)), \br
	\inf_{R \times \Gamma^\vt_D} \vtB &> 0.
\label{w13}	
\end{align}

Similarly, we suppose
\begin{align} 
	\vc{b}_\tau \in BC(R; C^2(\Gamma^\vB_D; R^3)), \ \partial_t \vc{b}_\tau \in BC(R; C^1(\Gamma^\vB_D; R^3)) \label{w14} \\ 
	b_\nu \in BC(R; C^2(\Gamma^\vB_N)), \ \partial_t b_\nu \in BC(R; C^1(\Gamma^\vB_N)).
\label{w15}	
	\end{align}
Moreover, we impose the necessary compatibility conditions
\begin{equation} \label{w15a}
\vc{b}_\tau \cdot \vc{n}|_{\Gamma^\vB_D} = 0,\ 
\mbox{and}\ \int_{\partial \Omega} b_\nu \ \D \sigma = 0 \ \mbox{if}\ \Gamma^{\vB}_D = \emptyset.
\end{equation}

Finally, we introduce the space 
\begin{equation} \label{w16}
H_{0, \sigma} = \left\{ \vc{b} \in L^2(\Omega; R^3) \ \Big|\ \Curl \vc{b} \in L^2(\Omega; R^3),\ 
\Div \vc{b} = 0, \vc{b} \times \vc{n}|_{\Gamma^\vB_D} = 0,\ \vc{b} \cdot \vc{n}|_{\Gamma^\vB_N} = 0    \right\}.
\end{equation}
The space $H_{0,\sigma}$ is a Hilbert space with the norm $\| \vc{b} \|^2_{H_0} = \| \Curl \vc{b} \|^2_{L^2(\Omega; R^3)} + \| \vc{b} \|^2_{L^2(\Omega; R^3)}$. The space $\mathcal{H}$ introduced in \eqref{p15} is 
a closed subspace of $H_{0, \sigma}$, and the following version of Poincar\' e inequality 
\begin{equation} \label{w17}
	\| \vc{b} \|_{W^{1,2}(\Omega; R^3)} \aleq \| \Curl \vc{b} \|_{L^2(\Omega; R^3)} 
	\ \mbox{holds for all}\ \vc{b}\in \mathcal{H}^\perp \cap H_{0, \sigma},
	\end{equation}
see Csat\' o, Kneuss and Rajendran \cite[Theorem 2.1]{CsKnRa}.

The specific extensions of the boundary data will be constructed in Section \ref{EE}.

\subsection{Weak solutions}

As we are interested in the long-time behaviour when the system ``forgets'' its initial state, the choice 
of initial data plays no role in the analysis. Accordingly, it is convenient to introduce the concept 
of global in time weak solutions defined for $t \in (T, \infty)$.

\begin{mdframed}[style=MyFrame]
	
	\begin{Definition} \label{Dw1} {\bf (Global in time weak solutions)}

A quantity $(\vr, \vt, \vu, \vB)$ is termed \emph{weak solution} of the compressible MHD system
\eqref{p1}--\eqref{p4}, with the boundary conditions \eqref{p6} -- \eqref{p11} in the 
time--space cylinder $(T, \infty) \times \Omega$
if the following holds:
\begin{itemize}
	\item{\bf Equation of continuity.}
	$\vr \in L^\infty_{\rm loc}(T, \infty; L^{\frac{5}{3}} (\Omega))$, $\vr \geq 0$, and the integral
	identity
\begin{equation} \label{w18}
	\int_T^\infty \intO{ \Big( \vr \partial_t \varphi + \vr \vu \cdot \Grad \varphi \Big) } \ dt = 0
\end{equation}
holds for any $\varphi \in C^1_c((T, \infty) \times \Ov{\Omega})$. 	
In addition, the renormalized version of \eqref{w18}
\begin{equation} \label{w19}
	\int_T^\infty \intO{ \left( b(\vr) \partial_t \varphi + b(\vr) \vu \cdot \Grad \varphi +
		\Big( b(\vr) - b'(\vr) \vr \Big) \Div \vu \varphi \right) } \dt = 0
\end{equation}
holds for any $\varphi \in C^1_c((T, \infty) \times \Ov{\Omega})$ and any $b \in C^1(R)$, $b' \in C_c(R)$.

\item{\bf Momentum equation.}

$\vr \vu \in L^\infty_{\rm loc}(T, \infty ; L^{\frac{5}{4}}(\Omega; R^3))$,
$\vu \in L^2_{\rm loc}(T, \infty; W^{1,2}(\Omega; R^3))$, $\vu|_{\Gamma^\vu_D} = 0$,
$\vu \cdot \vc{n}|_{\Gamma^\vu_N} = 0$,
and the integral identity
\begin{align}
	\int_T^\infty &\intO{ \Big( \vr \vu \cdot \partial_t \bfphi + \vr \vu \otimes \vu : \Grad \bfphi + 
		(\vr \vu \times \bfomega) \cdot \bfphi +
		p(\vr, \vt) \Div \bfphi \Big) } \br &= \int_0^\tau \intO{ \mathbb{S}(\vt, \Grad \vu) : \Grad \bfphi } \dt -
	\int_T^\infty \intO{ \left( \vB \otimes \vB - \frac{1}{2} |\vB |^2 \mathbb{I} \right) : \Grad \bfphi } \dt \br &-
	\int_T^\infty \intO{ \vr \Grad M \cdot \bfphi } \dt  	
	\label{w20}
\end{align}
for any $\bfphi \in C^1_c((T, \infty) \times \Ov{\Omega}; R^3)$,
$\bfphi|_{\Gamma^\vu_D} = 0$, $\bfphi \cdot \vc{n}|_{\Gamma^\vu_N} = 0$.

\item {\bf Induction equation.}
$\vB \in L^\infty_{\rm loc}((T, \infty); L^2(\Omega; R^3))$,
\begin{equation} \label{w21}
	\Div \vc{B}(\tau, \cdot) = 0 \ \mbox{for any}\ \tau \in (T, \infty),
	\end{equation} 
\begin{equation} \label{w22}
	(\vB - \bB) \in L^2_{\rm loc}(T, \infty; H_{0, \sigma}(\Omega; R^3)).
	\end{equation}
The integral identity
\begin{align}
	\int_T^\infty \intO{ \Big( \vB \cdot \partial_t \bfphi - (\vB \times \vu) \cdot \Curl \bfphi -
		\zeta (\vt) \Curl \ \vB \cdot \Curl \ \bfphi \Big) } \dt  = 0
	\label{w23}
\end{align}	
holds for any $\bfphi \in C^1_c((T, \infty) \times \Ov{\Omega}; R^3)$,
\begin{equation} \label{w24}
\bfphi \times \vc{n}|_{\Gamma^\vB_D} = 0,\ \bfphi \cdot \vc{n}|_{\Gamma^\vB_N} = 0.
\end{equation}

\item {\bf Entropy inequality.}
$\vt \in L^\infty_{\rm loc}(T, \infty; L^4(\Omega)) \cap L^2_{\rm loc}(T, \infty; 
W^{1,2}(\Omega))$, $\vt > 0$ a.a. in $(T, \infty) \times \Omega$,
$\log (\vt) \in L^2_{\rm loc}(T, \infty; W^{1,2}(\Omega))$, 
\[
\vt|_{\Gamma^\vt_D} = \vtB.
\]
The integral inequality
\begin{align}
	\int_T^\infty &\intO{ \left( \vr s (\vr, \vt) \partial_t \varphi + \vr s (\vr, \vt) \vu \cdot \Grad \varphi + \frac{\vc{q} (\vt, \Grad \vt) }{\vt} \cdot
		\Grad \varphi \right) } \dt \br & \leq  - \int_T^\infty \intO{ \frac{\varphi}{\vt}
		\left( \mathbb{S}(\vt, \Grad \vu) : \Grad \vu - \frac{\vc{q} (\vt, \Grad \vt) \cdot \Grad \vt}{\vt} + \zeta (\vt) |\Curl \vB|^2 \right) }
	\dt	
	\label{w25}	
\end{align}
holds for any $\varphi \in C^1_c((T, \infty) \times \Ov{\Omega})$, $\varphi \geq 0$, $\varphi|_{\Gamma^\vt_D} = 0$.

\item {\bf Ballistic energy inequality.}
The inequality
\begin{align}
	&\int_T^\infty \partial_t \psi	\intO{ \left( \frac{1}{2} \vr |\vu|^2 + \vr e(\vr, \vt) + \frac{1}{2} |\vB |^2 - \tvt \vr s(\vr, \vt) - \bB \cdot \vB\right) } \dt \br
	\quad&
	- \int_T^\infty \psi \intO{ 	\frac{\tvt}{\vt} \left( \mathbb{S}(\vt, \Grad \vu) : \Grad \vu - \frac{\vc{q}(\vt, \Grad \vt) \cdot \Grad \vt}{\vt} + \zeta (\vt) | \Curl \vB |^2 \right) } \dt \br
	&\quad \geq  \int_T^\infty \psi \intO{\left( \vr s (\vr, \vt) \partial_t \tvt + \vr s (\vr, \vt) \vu \cdot \Grad \tvt + \frac{\vc{q}(\vt, \Grad \vt)}{ \vt} \cdot \Grad \tvt                    \right)      } \dt\br
	&\quad + \int_T^\infty \psi \intO{  \Big( \vB \cdot \partial_t \bB - (\vB \times \vu) \cdot \Curl \bB -
		\zeta (\vt) \Curl \ \vB \cdot \Curl \bB \Big)    } \dt \br 	
	&\quad -  \int_T^\infty \psi \intO{ \vr \Grad M \cdot \vu } \dt 
	\label{w26}
\end{align}
holds for any $\psi \in C^1_c(T, \infty)$, $\psi \geq 0$, and any
$\tvt \in C^1_c((T, \infty) \times \Ov{\Omega})$, $\tvt > 0$, $\tvt|_{\Gamma^\vt_D} = \vtB$.

	\end{itemize}		
		
		\end{Definition}

	\end{mdframed}

\begin{Remark} \label{Rw1}
	
	Note carefully that the ballistic energy inequality \eqref{w26} remains valid if we replace
	$\bB$ by any other extension $\tvB$ such that
	\[
	\Div \tvB = 0,\ \tvB \times \vc{n}|_{\Gamma^\vB_D} =
	\bB \times \vc{n}|_{\Gamma^\vB_D},\ \tvB \cdot \vc{n}|_{\Gamma^\vB_N} =
	\bB \cdot \vc{n}|_{\Gamma^\vB_N},
	\]
	respectively. Indeed the difference $\bB - \tvB$ becomes an eligible test function
	for the weak formulation of the induction equation \eqref{w23}, \eqref{w24}.  
	
	\end{Remark}

The existence of global--in--time weak solutions as well as the weak--strong uniqueness property for any finite energy initial data was shown in \cite{FeGwKwSG} under more restrictive hypotheses on the boundary data.

\section{Bounded absorbing set}
\label{AS}

We are ready to state the first main result of the present paper concerning the existence of a bounded absorbing set. It is convenient to suppose that both the potential $G$ as well as the 
boundary data $\vtB$, $\vc{b}_\tau$, and
$b_\nu$ are defined for all $t \in R$. Accordingly, we introduce 
\begin{align} 
\| (\mbox{data}) \| &= m_0 + m_0^{-1} + \|G \|_{W^{1,\infty}(R \times \Omega)} + |\bfomega| + 
\| \vtB^{-1} \|_{L^\infty(R \times \partial \Gamma^\vt_D)} \br
&+ \sup_{t \in R} \| \vtB (t, \cdot) \|_{W^{2, \infty}(\Gamma^\vt_D)} + 
\sup_{t \in R} \| \partial_t \vtB \|_{W^{1,\infty}(\Gamma^\vt_D )}
 \br
&+ \sup_{t \in R} \| \vc{b}_\tau (t, \cdot) \|_{W^{2, \infty}(\Gamma^\vB_D; R^3)} + 
\sup_{t \in R} \| b_\nu \|_{W^{2,\infty}(\Gamma^\vB_N )} \br &+
\sup_{t \in R} \| \partial_t \vc{b}_\tau (t, \cdot) \|_{W^{1, \infty}(\Gamma^\vB_D; R^3)} + 
\sup_{t \in R} \| \partial_t b_\nu \|_{W^{1,\infty}(\Gamma^\vB_N )}.
\label{AS1}
\end{align}
In what follows, we shall use the symbol $c(\nb)$ to denote a generic positive function of $\nb$ bounded for bounded 
arguments.

\begin{mdframed}[style=MyFrame]
	
	\begin{Theorem} \label{TAS1} {\bf [Bounded absorbing set] }
		
		Let $\Omega \subset R^3$ be a bounded domain of class $C^{2 + \nu}$, the boundary of which admits the decomposition \eqref{p5}. Let the pressure $p$ be related to the internal energy $e$ 
		through the equations of state \eqref{eosT}, \eqref{eosC}, \eqref{eosM}, where 
		the functions $P$ and $\mathcal{S}$ satisfy \eqref{w3}, \eqref{w7}.
		Let the transport coefficients $\mu$, $\eta$, $\kappa$, and $\zeta$
		be continuously differentiable functions of the temperature satisfying \eqref{w10}--\eqref{w12}. Finally, let $G \in BC^1(R \times \Ov{\Omega})$ and let the boundary data belong to the class \eqref{w13}--\eqref{w15a}, where  
			\begin{equation} \label{AS3}
			\Gamma^\vu_D \ne \emptyset,\ \Gamma^\vt_D \ne \emptyset.  
			\end{equation}
		If, in addition, the boundary magnetic field $b_\nu(\tau_n, \cdot)$, $\vc{b}_\tau(\tau_n, \cdot)$ is not stationary for some sequence of times $\tau_n \to \infty$, we assume 
		\begin{equation} \label{AS3a}
			\Gamma^\vB_D \subset \Gamma^\vu_D.
			\end{equation}
		
		Then there exists a positive constant $\mathcal{E}_\infty (\nb)$ depending only on the amplitude 
		of the data specified in \eqref{AS1} such that following holds. For any weak solution $(\vr, \vt, \vu, \vB)$ of 
		the compressible MHD system in $(T, \infty) \times \Omega$ satisfying 
		\[
		\intO{\vr } = m_0,\ \intO{ \vB \cdot \vc{h} } = 0 \ \mbox{for all}\ \vc{h} \in \mathcal{H}(\Omega), 
		\]
		there exists a time $\tau > 0$ such that 
		\[
		\intO{ E(\vr, \vt, \vu, \vB) (t, \cdot) } < \mathcal{E}_\infty 
		\ \mbox{for all}\ t > T + \tau.
		\]
		Moreover, the length of the time $\tau$ depends only on $\nb$ and on the ``initial energy''
		\[
		\mathcal{E}_T \equiv \limsup_{t \to T-} \intO{ E(\vr, \vt, \vu, \vB) (t, \cdot) }.
		\] 
		\end{Theorem}

	\end{mdframed}

\bigskip

The rest of this section is devoted to the proof of Theorem \ref{TAS1}. A similar result in the context of the 
Rayleigh--B\' enard convection problem was shown in \cite{FeiSwGw}. The influence of the magnetic field on the 
fluid motion, however, requires essential modifications of the proof presented in \cite{FeiSwGw}. In particular, 
we construct a two component extension of the boundary magnetic field, where the first component 
is solenoidal, irrotational and satisfies the Neumann boundary condition, while the second one is solenoidal and small in 
a suitable Lebesgue norm in $\Omega$ with bounded rotation, see Section \ref{E} below. 

As already pointed out, the condition $\Gamma^\vt_D \ne \emptyset$ is essential, and, in fact, necessary for the 
existence of a universal absorbing set. Indeed we show an example of a flow driven by an external magnetic field 
with thermally insulated boundary, where the total energy becomes unbounded for $t \to \infty$, see Section \ref{AP}.
 
\subsection{Extension of the boundary data}
\label{EE}

We start by specifying the extension of the boundary data for the temperature and the magnetic field. 

\subsubsection{Extension of the boundary temperature}

Given $\vtB$ we consider its harmonic extension $\tvt$, 
\begin{equation} \label{AS4}
\Del \tvt (\tau, \cdot) = 0,\ \tvt(\tau, \cdot)|_{\Gamma^\vt_D} = \vtB(\tau, \cdot),\ 
\Grad \tvt (\tau, \cdot) \cdot \vc{n}|_{\Gamma^\vt_N} = 0,\ \tau \in R.	
	\end{equation}
We shall use the same symbol $\vtB$ for the boundary data and their extension $\tvt$ hereafter.

\subsubsection{Extension of the boundary magnetic field}
\label{E}

If the boundary data $b_\nu$, $\vc{b}_\tau$ are stationary, we simply consider 
the solenoidal irrotational extension $\vc{B}_B$ claimed in Definition \ref{Dew1}.

Extension of the \emph{non--stationary} boundary magnetic field is more delicate. First, we find an extension $\tvB_N$ satisfying 
\begin{equation} \label{AS5}
\Div \tvB_N = 0,\ \Curl \tvB_N = 0,\ \tvB_N \cdot \vc{n}|_{\Gamma^\vB_N} = b_\nu,\ \tvB_N \times \vc{n}|_{\Gamma^\vB_D} = 0.	
	\end{equation}
Indeed we may consider $\tvB_N = \Grad \Phi_N$, where $\Phi_N$ is the unique solution of the mixed 
boundary value problem 
\[
\Del \Phi_N = 0,\ \Phi_N|_{\Gamma^\vB_D} = 0,\ \Grad \Phi_N \cdot \vc{n}|_{\Gamma^\vB_N} = b_\nu. 
\]

To find a suitable extension of the Dirichlet data $\vc{b}_\tau$, we introduce the so--called \emph{Bogovskii operator} $\mathcal{B}$ enjoying the following properties, see e.g. Galdi \cite{GALN}, Gei{\ss}ert,  Heck, and Hieber \cite{GEHEHI}: 
\begin{itemize}
\item 
$\mathcal{B}$ is a bounded linear operator mapping $L^q_0(Q)$ to $W^{1,q}_0(Q; R^3)$, where 
$Q$ is a bounded Lipschitz domain in $R^d$ and the symbol $L^q_0$ denotes the space of $L^q-$integrable 
functions with zero mean:	
\begin{align} \label{AS6}
	\Div \mathcal{B}[f] = f \ \mbox{in}\ Q \ \mbox{for any}\ f \in L^q(Q), \ \int_Q f = 0, \
	\mathcal{B}[f]|_{\partial Q} = 0, \br 
	\| \mathcal{B}[f]  \|_{W^{1,q}(Q; R^3)} \aleq \| f \|_{L^q(Q)},\ 1 < q < \infty.
	\end{align}
\item If, moreover, $f = \Div \vc{g}$, $\vc{g} \in L^r(Q; R^3)$, $\vc{g} \cdot \vc{n}|_{\partial Q} = 0$, then 
\begin{equation} \label{AS7}
	\| \mathcal{B}[\Div \vc{g} ] \|_{L^r(Q; R^3)} \aleq \| \vc{g} \|_{L^r(Q; R^3)},\ 1 < r < \infty.
	\end{equation}

	\end{itemize}

Now, consider the boundary vector 
\begin{equation} \label{AS8}
\vB_D = \vc{n} \times \vc{b}_\tau \ \mbox{on}\ \Gamma^\vB_D \ \Rightarrow 
\ 
\vB_D \times \vc{n} = \vc{b}_\tau \ \mbox{on}\ \Gamma^\vB_D. 
\end{equation}
First, let $\tvB_D$ be an arbitrary, not necessarily solenoidal, Lipschitz extension of $\vB_D$ inside $\Omega$. 
Consider a cut-off function 
\[
\chi \in \DC [0, 1), \ 0 \leq \chi \leq 1,\ 
\chi (z) = 1 \ \mbox{for}\ z \in [0, d), 0 < d < 1,
\] 
together with 
\[
\vB_{\delta, D}(t,x) = 
\chi \left( \frac{ {\rm dist}[x, \Gamma^\vB_D ] }{\delta} \right) \tvB_D (t,x), \ \delta > 0.
\]
It follows from \eqref{AS8} that 
\begin{align}
	\vB_{\delta,D} \times \vc{n} &= \vc{b}_\tau \ \mbox{on}\ \Gamma^\vB_D, \br 
	{\rm supp}[ \vB_{\delta, D} ] &\subset \mathcal{U}_{\delta_0} (\Gamma^\vB_D) 
	\ \mbox{for all}\ 0 < \delta \leq \delta_0, 
\label{AS9}
\end{align}
where $\mathcal{U}_{\delta_0}$ denotes the $\delta_0-$neighbourhood.

Next, fix $\delta_0$ so that $\mathcal{U}_{\delta_0}(\Gamma^\vB_D) \cap \Gamma^\vB_N = \emptyset$ and consider
\begin{equation} \label{AS10}
	\tvB_{\delta, D} = \left\{ \begin{array}{l} \vB_{\delta,D} - \mathcal{B}_{\delta_0} [\Div \vB_{\delta,D} ]
		\ \mbox{in}\  \mathcal{U}_{\delta_0} ( \Gamma^\vB_D ) \cap \Omega,
		\\ \\ 0 \ \mbox{in}\ \Omega \setminus \mathcal{U}_{\delta_0} ( \Gamma^\vB_D ) \end{array} \right., 
\end{equation}
where $\mathcal{B}_{\delta_0}$ is the Bogovskii operator defined on each component of the open set 
$\mathcal{U}_{\delta_0} ( \Gamma^{\vB}_D ) \cap \Omega$. It is easy to check that $\tvB_{\delta,D}$ enjoys the following properties:
\begin{align} 
\tvB_{\delta, D} \times \vc{n}|_{\Gamma^\vB_D} &= \vc{b}_\tau , \
\tvB_{\delta, D}|_{\Gamma^\vB_N} = 0,\br 
\Div \tvB_{\delta, D} &= 0 \ \mbox{in}\ \Omega.
	\label{AS11}
	\end{align}
In addition, as a consequence of boundedness of the Bogovskii operator \eqref{AS6}, 
\begin{equation} \label{AS12}
\| \tvB_{\delta, D} \|_{W^{1,p}(\Omega; R^3)} \leq c(p, \delta) \| \vc{b}_\tau \|_{W^{1, \infty}(\Gamma^\vB_D)},\ 
1 < p < \infty.	
	\end{equation}
Note we might have $c(p, \delta) \to \infty$ for $\delta \to 0$.
Finally, by virtue of \eqref{AS7}, 
\begin{equation} \label{AS13}
\| \tvB_{\delta, D} \|_{L^r(\Omega; R^3)} \leq o(r, \delta) \| \vc{b}_\tau \|_{W^{1, \infty}(\Gamma^\vB_D)},\ \mbox{where}\ o(r, \delta) \to 0 \ \mbox{as}\ \delta \to 0\ \mbox{for any}\ 1 < r < \infty.
	\end{equation} 

We consider the extension of the boundary data in the form 
\begin{equation} \label{AS14}
	\bB = \vB_{\delta,B} = \tvB_N + \tvB_{\delta, D}
	\end{equation} 
for a suitable $0 < \delta < \delta_0$ to be fixed below. By virtue of \eqref{AS5},  
\begin{equation} \label{AS15}
\Div \vB_{\delta,B} = 0,\ \Curl \vB_{\delta,B} = \Curl \tvB_{\delta, D},\ 
\ \vB_{\delta,B} \times \vc{n}|_{\Gamma^\vB_D} = \vc{b}_\tau,\ 
\vB_{\delta,B} \cdot \vc{n}|_{\Gamma^\vB_N} = b_\nu .
	\end{equation}

\begin{Remark} \label{RAS1}
	
	A similar ansatz for the function $\tvB_{\delta, D}$ was used by Zhang and Zhao \cite{ZhanZhao} based on the construction
	proposed by Alekseev \cite{Aleks}.
	
	\end{Remark}

\subsection{Ballistic energy inequality}

The proof of Theorem \ref{TAS1} is based on the ballistic energy inequality \eqref{w26} evaluated for 
$\tvt = \vtB$, $\bB = \vc{B}_{\delta, B}$ specified in Section \ref{EE}. In accordance with \eqref{AS15}, 
we get:
\begin{align}
	&\frac{\D }{\dt}	\intO{ \left( \frac{1}{2} \vr |\vu|^2 + \vr e(\vr, \vt) + \frac{1}{2} |\vB |^2 - \vtB \vr s(\vr, \vt) -  \vc{B}_{\delta, B} \cdot \vc{B}  \right) }  \br
	\quad&
	+ \intO{ 	\frac{\vtB}{\vt} \left( \mathbb{S}(\vt, \Grad \vu) : \Grad \vu - \frac{\vc{q}(\vt, \Grad \vt) \cdot \Grad \vt}{\vt} + \zeta (\vt) | \Curl \vB |^2 \right) }  \br
	&\quad \leq  - \intO{\left( \vr s (\vr, \vt) \partial_t \vtB + \vr s (\vr, \vt) \vu \cdot \Grad \vtB + \frac{\vc{q}(\vt, \Grad \vt)}{ \vt} \cdot \Grad \vtB                    \right)      }  \br
	&\quad - \intO{  \Big( \vB \cdot \partial_t \vB_{\delta,B} - (\vB \times \vu) \cdot \Curl \tvB_{\delta, D} -
		\zeta(\vt) \Curl \ \vB \cdot \Curl \tvB_{\delta, D} \Big)    }  \br 	
	&\quad + \intO{ \vr \Grad M \cdot \vu }. 
	\label{AS16}
\end{align}
in $\mathcal{D}' (T, \infty)$. Note that we may suppose $\Curl \tvB_{\delta, D} = 0$ if the boundary 
magnetic field is stationary.

\begin{Remark} \label{RB1} 
	
	Strictly speaking, the pointwise values of the ballistic energy  are defined only for a.a. $\tau \in (T, \infty)$. However, thanks to the inequality \eqref{AS16}, we may identify 
	the ballistic energy with its c\` agl\` ad representative 
\begin{align}
&\intO{ \left( \frac{1}{2} \vr |\vu|^2 + \vr e(\vr, \vt) + \frac{1}{2} |\vB |^2 - \vtB \vr s(\vr, \vt) -  \vc{B}_{\delta, B} \cdot \vc{B}  \right) (\tau, \cdot) } \br &\quad = {\rm ess} \lim_{t \to \tau -}
\intO{ \left( \frac{1}{2} \vr |\vu|^2 + \vr e(\vr, \vt) + \frac{1}{2} |\vB |^2 - \vtB \vr s(\vr, \vt) -  \vc{B}_{\delta, B} \cdot \vc{B}  \right) (t, \cdot) } 
\end{align}
defined for any $\tau > T$.  
	
\end{Remark}

Our goal is to show that all terms on the right--hand side of \eqref{AS16} can be controlled by the dissipation 
rate appearing on the left--hand side. This is relatively straightforward for the integrals independent of the 
density. 

\subsubsection{Coercivity of the dissipation rate}

It follows from coercivity of the transport coefficients stated in the hypotheses \eqref{w10}--\eqref{w12}, 
and hypothesis \eqref{AS3}  $\Gamma^\vu_D \ne \emptyset$ that 
\begin{align}
 &\intO{ 	\frac{\vtB}{\vt} \left( \mathbb{S}(\vt, \Grad \vu) : \Grad \vu - \frac{\vc{q}(\vt, \Grad \vt) \cdot \Grad \vt}{\vt} + \zeta (\vt) | \Curl \vB |^2 \right) } \br &\quad \ageq
 \inf \vtB \left[ \| \vu \|^2_{W^{1,2}(\Omega; R^3)} + \| \Grad \vt^{\frac{\beta}{2}} \|^2_{L^2(\Omega; R^3)} +
 \| \Grad \log \vt \|^2_{L^2(\Omega; R^3)} + \| \Curl \vB \|^2_{L^2(\Omega; R^3)} \right]. 	
\nonumber
	\end{align} 
In addition, as $\Gamma^\vt_D \ne \emptyset$, we may apply Poincar\' e inequality obtaining
\[
 \| \Grad \vt^{\frac{\beta}{2}} \|^2_{L^2(\Omega; R^3)} +
\| \Grad \log \vt \|^2_{L^2(\Omega; R^3)} \ageq  \| \vt^{\frac{\beta}{2}} \|^2_{W^{1,2}(\Omega)} +
\| \log \vt \|^2_{W^{1,2}(\Omega)} - c(\nb).
\]
Similarly, by virtue of inequality \eqref{w17},
\[
 \| \Curl \vB \|^2_{L^2(\Omega; R^3)} \ageq \| \vB \|_{W^{1,2}(\Omega; R^3)}^2 - c(\nb).
\]
Indeed we may use the extension $\vB_{\delta_0, B}$ of the boundary field constructed in Section 
\ref{E}, together with its projection $\widehat{\vB}_{\delta_0, B}$ on the space $\mathcal{H}^\perp$
to deduce
\begin{align}
\| \Curl \vB \|^2_{L^2(\Omega; R^3)} &= \| \Curl (\vB - \widehat{\vB}_{\delta_0, B}) \|^2_{L^2(\Omega; R^3)} - c_1(\nb) \br 
&\ageq \| \vB - \widehat{\vB}_{\delta_0, B} \|^2_{W^{1,2}(\Omega; R^3)} - c_1(\nb) \ageq 
\| \vB  \|^2_{W^{1,2}(\Omega; R^3)} - c_2(\nb). 
\nonumber
\end{align}
It is important to notice that the above estimate is independent of the value of the parameter $\delta$ to be fixed
later in the course of the proof.

Summing up the previous estimates we may infer that
\begin{align}
	&\left[ \| \vu \|^2_{W^{1,2}(\Omega; R^3)} + \|  \vt^{\frac{\beta}{2}} \|^2_{W^{1,2}(\Omega)} +
	\| \log \vt \|^2_{W^{1,2}(\Omega)} + \| \vB \|^2_{W^{1,2}(\Omega; R^3)} \right] \br &\quad 
	\leq c(\nb)\left[ 1 + \intO{ 	\frac{\vtB}{\vt} \left( \mathbb{S}(\vt, \Grad \vu) : \Grad \vu - \frac{\vc{q}(\vt, \Grad \vt) \cdot \Grad \vt}{\vt} + \zeta (\vt) | \Curl \vB |^2 \right) } \right].	
	\label{AS17}
\end{align}
We point out that hypothesis \eqref{AS3} plays a crucial role in this step. 

\subsubsection{Estimates of the integrals containing the magnetic field}

We use again hypothesis \eqref{AS3} together with \eqref{AS10} to perform integration by parts 
\[
\intO{ (\vB \times \vu) \cdot \Curl \tvB_{\delta, D} } = 
\intO{ \Curl  (\vB \times \vu) \cdot \tvB_{\delta, D} }, 
\]
where 
\begin{align}
\left| \intO{ \Curl  (\vB \times \vu) \cdot \tvB_{\delta, D} } \right| \leq 
\| \tvB_{\delta, D} \|_{L^3(\Omega; R^3)} \| \Curl  (\vB \times \vu) \|_{L^{\frac{3}{2}}(\Omega; R^3)} \br 
\aleq \| \tvB_{\delta, D} \|_{L^3(\Omega; R^3)} \left( \| \vB \|_{L^6(\Omega; R^3)} \| \Grad \vu \|_{L^2(\Omega; 
	R^{3 \times 3})} + \| \Grad \vB \|_{L^2(\Omega; R^{3 \times 3})} \| \vu \|_{L^6(\Omega; 
	R^{3})}    \right).
\nonumber
\end{align}
Consequently, by virtue of \eqref{AS13} and the Sobolev embedding $W^{1,2} \hookrightarrow L^6$, 
\begin{equation} \label{AS18}
\left| \intO{ (\vB \times \vu) \cdot \Curl \tvB_{\delta, D} } \right| \leq 
o(\delta, \nb ) \left( \| \vu \|_{W^{1,2}(\Omega; R^3)}^2 + \| \vB \|_{W^{1,2}(\Omega; R^3)}^2 \right),	
	\end{equation}
where $o(\delta, \nb) \to 0$ as $\delta \to 0$ uniformly for $t \in R$.

Similarly, using \eqref{AS12} for $p = 2$, we get 
\begin{equation} \label{AS19}
	\left| \intO{ \vB \cdot \partial_t \vB_{\delta, B} } \right| \leq 
\delta \| \vB \|_{L^2(\Omega; R^3)}^2 + c(\delta, \nb) \ \mbox{for any}\ \delta > 0,	
	\end{equation}
and, by the same token, 
\begin{align}
&\left| \intO{ \zeta(\vt) \Curl \ \vB \cdot \Curl \tvB_{\delta, D} } \right| \br &\quad \leq 
\delta \| \Curl \vB \|^2_{L^2(\Omega; R^3)} + c(\delta, \nb, r) \| \zeta(\vt) \|^2_{L^r(\Omega)} 
\ \mbox{for some}\ r > 2.
\nonumber
\end{align}
Furthermore, as $\frac{\beta}{2} > 3 (> 2)$ and the growth of $\zeta$ is restricted by \eqref{w12}, 
we may infer that 
\begin{equation} \label{AS20}
\left| \intO{ \zeta(\vt) \Curl \ \vB \cdot \Curl \tvB_{\delta, D} } \right| 
\leq \delta \left(  \| \Curl \vB \|^2_{L^2(\Omega; R^3)}  + 
\| \vt^{\frac{\beta}{2}} \|^2_{W^{1,2}(\Omega)}  \right) + c(\delta, \nb) 
\end{equation}
for any $0 < \delta \leq \delta_0$.	

Thus fixing $\delta > 0$ sufficiently small we may use \eqref{AS17}, together with the bounds 
\eqref{AS18}--\eqref{AS20} to rewrite inequality \eqref{AS16} in the form 
\begin{align}
	&\frac{\D }{\dt}	\intO{ \left( \frac{1}{2} \vr |\vu|^2 + \vr e(\vr, \vt) + \frac{1}{2} |\vB |^2 - \vtB \vr s(\vr, \vt) -  \vc{B}_{B} \cdot \vc{B}  \right) }  \br
	\quad&
	+ \frac{3}{4} \intO{ 	\frac{\vtB}{\vt} \left( \mathbb{S}(\vt, \Grad \vu) : \Grad \vu - \frac{\vc{q}(\vt, \Grad \vt) \cdot \Grad \vt}{\vt} + \zeta (\vt) | \Curl \vB |^2 \right) }  \br
	&\quad \leq  - \intO{\left( \vr s (\vr, \vt) \partial_t \vtB + \vr s (\vr, \vt) \vu \cdot \Grad \vtB + \frac{\vc{q}(\vt, \Grad \vt)}{ \vt} \cdot \Grad \vtB                    \right)      }  \br
	&\quad + \intO{ \vr \Grad M \cdot \vu } + c(\nb),
	\label{AS21}
\end{align}
where we have set $\bB = \vB_{\delta, B}$.

\subsubsection{Estimates of the integrals containing the density and temperature}

First, write
\begin{equation} \label{AS22}
\intO{ \vr \Grad M \cdot \vu } = \frac{\D }{\dt} \intO{ \vr M } - \intO{ \vr \partial_t M }.
\end{equation}
Next, in accordance with the choice of $\vtB$ in \eqref{AS4}, 
\begin{equation} \label{AS23}
\intO{ \frac{\vc{q}(\vt, \Grad \vt)}{ \vt} \cdot \Grad \vtB } = 
- \intO{ \frac{\kappa (\vt)}{\vt} \Grad \vt \cdot \Grad \vtB } = - \int_{\Gamma^\vt_D} \mathcal{K} (\vtB) \Grad \vtB \cdot \vc{n}\ \D \sigma,  	
	\end{equation}
where $\mathcal{K}'(\vt) = \frac{\kappa (\vt)}{\vt}$.

In view of \eqref{AS22}, \eqref{AS23}, the inequality \eqref{AS21} can be rewritten in the form 
\begin{align}
	&\frac{\D }{\dt}	\intO{ \left( \frac{1}{2} \vr |\vu|^2 + \vr e(\vr, \vt) + \frac{1}{2} |\vB |^2 - \vtB \vr s(\vr, \vt) -  \vc{B}_{B} \cdot \vc{B} - \vr M \right) }  \br
	\quad&
	+ \frac{3}{4} \intO{ 	\frac{\vtB}{\vt} \left( \mathbb{S}(\vt, \Grad \vu) : \Grad \vu - \frac{\vc{q}(\vt, \Grad \vt) \cdot \Grad \vt}{\vt} + \zeta (\vt) | \Curl \vB |^2 \right) }  \br
	&\quad \leq  - \intO{\Big( \vr s (\vr, \vt) \partial_t \vtB + \vr s (\vr, \vt) \vu \cdot \Grad \vtB          \Big)      }   + c(\nb).
	\label{AS24}
\end{align}

The integrals containing entropy can be handled in the same way as in \cite[Section 4.4]{FeiSwGw}. In accordance 
with hypothesis \eqref{w5}, 
\[
\vr s(\vr, \vt) = \vr \mathcal{S} \left( \frac{\vr}{\vt^{\frac{3}{2}}} \right) + \frac{4a}{3} \vt^3.
\]
On the one hand, as $s$ complies with the Third law of thermodynamics enforced by the hypotheses \eqref{w6}, \eqref{w7}, 
we have
\begin{equation} \label{AS25}
	\vr \mathcal{S} \left( \frac{\vr}{\vt^{\frac{3}{2}}} \right) \leq \vr \mathcal{S}(r) \ \mbox{provided} \ \frac{\vr}{\vt^{\frac{3}{2}}} 
	\geq r,\ \mbox{where}\ \mathcal{S}(r) \to 0 \ \mbox{as}\ r \to \infty.
\end{equation}
On the other hand, if $\vr < r \vt^{\frac{3}{2}}$, we get 
\begin{equation} \label{AS26}
	0 \leq \vr \mathcal{S} \left( \frac{\vr}{\vt^{\frac{3}{2}}} \right) 
	\leq c \left(1 + r \vt^{\frac{3}{2}} \left[ \log^+ (r \vt^{\frac{3}{2}} ) + \log^+(\vt) \right] \right), 
\end{equation}
cf. also \cite[Chapter 12, formula (12.35)]{FeiNovOpen}.

Thus we may conclude, exactly as in \cite[Section 4.4]{FeiSwGw}, 
\begin{align}
	&\frac{\D }{\dt}	\intO{ \left( \frac{1}{2} \vr |\vu|^2 + \vr e(\vr, \vt) + \frac{1}{2} |\vB |^2 - \vtB \vr s(\vr, \vt) -  \vc{B}_{B} \cdot \vc{B} - \vr M \right) }  \br
	\quad&
	+ \frac{1}{4} \intO{ 	\frac{\vtB}{\vt} \left( \mathbb{S}(\vt, \Grad \vu) : \Grad \vu - \frac{\vc{q}(\vt, \Grad \vt) \cdot \Grad \vt}{\vt} + \zeta (\vt) | \Curl \vB |^2 \right) }  \br
	&\quad \leq  \mathcal{S}(r) c_1(\nb) \intO{ \vr |\vu|       }   + c_2(r,\nb), \ \mathcal{S}(r) \to 0 
	\ \mbox{as}\ r \to \infty, 
	\label{AS27}
\end{align}
where $c_2(r, \nb) \to \infty$ as $r \to \infty$.

\subsection{Pressure estimates}

The integral on the right--hand side of \eqref{AS27} contains $\vr$ and as such is not directly controllable by 
the dissipation rate. Following the same strategy as in \cite{FeiSwGw} we evoke the \emph{pressure estimates} 
obtained by using the quantity 
\[
\bfphi = \psi (t) \mathcal{B} \left[ \vr^\alpha - \frac{1}{|\Omega|} \intO{ \vr^\alpha } \right]
\]
as a test function in the momentum balance \eqref{w25}, where $\mathcal{B}$ is the Bogovskii operator 
introduced in Section \ref{E}. As a consequence
of \eqref{AS6} and the Sobolev embedding $W^{1,p} \hookrightarrow L^\infty$, $p > 3$,
\begin{equation} \label{AS28}
 \left| \mathcal{B} \left[ \vr^\alpha - \frac{1}{|\Omega|} \intO{ \vr^\alpha } \right] \right| \leq c(m_0)
 \ \mbox{as soon as}\ \alpha < \frac{1}{3}.
\end{equation}

Following step by step the arguments of \cite[Section 4.6.1]{FeiSwGw} we deduce the inequality 
\begin{align}
\int_{\tau}^{\tau + 1} &\intO{ \vr^{\frac{5}{3} + \alpha} } \dt \leq c(m_0) \left[ 1 + 
\int_\tau^{\tau + 1} \| \vt \|^4_{L^4 (\Omega) }\dt    \right. \br
&\left( 1 + \sup_{t \in (\tau, \tau + 1)} \| \vr (t, \cdot) \|_{L^{\frac{5}{3}}(\Omega)} \right) 	
\int_{\tau}^{\tau+1} \| \vu \|^2_{W^{1,2}(\Omega; R^3)} \dt + \sup_{t \in (\tau, \tau + 1)} 
\intO{ \vr |\vu | } \br 
&+ \left. \int_{\tau}^{\tau + 1} \intO{ |\Curl \vB \times \vB | } \dt     \right],\ \alpha = \frac{1}{15},
	\label{AS29}
	\end{align}
which is nothing other than \cite[Section 4, formula (4.28)]{FeiSwGw} augmented by the contribution of the 
Lorentz force. 

\subsection{Proof of Theorem \ref{TAS1}}

The estimates \eqref{AS27}, \eqref{AS29} are the main tools for proving Theorem \ref{TAS1}. Denote 
\[
F(\vr, \vt, \vu, \vc{B} ) =  \frac{1}{2} \vr |\vu|^2 + \vr e(\vr, \vt) + \frac{1}{2} |\vB |^2 - \vtB \vr s(\vr, \vt) -  \vc{B}_{B} \cdot \vc{B} - \vr M.
\]
It follows from the structural restrictions imposed through the hypotheses \eqref{w1}--\eqref{w7} that 
\begin{equation} \label{AS30}
\frac{1}{\lambda} E(\vr, \vt, \vu, \vB) - c_1 (\lambda, \nb) \leq 
F(\vr, \vt, \vu, \vc{B} ) \leq \lambda E(\vr, \vt, \vu, \vB) + c_2 (\lambda, \nb)	
	\end{equation}
for any $\lambda > 1$. Similarly to \cite[Section 4.6]{FeiSwGw}, we show the following result. 

\begin{Lemma} \label{LAS1}
Under the hypotheses of Theorem \ref{TAS1}, 	
suppose 
\begin{equation} \label{AS31}
\intO{ F(\vr, \vt, \vu, \vc{B} ) (\tau, \cdot) } - 
\intO{ F(\vr, \vt, \vu, \vc{B} ) (\tau + 1, \cdot) } \leq K 	
	\end{equation} 
for some $\tau > T$, $K \geq 0$. 

Then there is $L = L(K, \nb)$ such that
\begin{equation} \label{AS32}
	\sup_{t \in (\tau, \tau + 1)} \intO{F (\vr, \vt, \vu, \vB)(t, \cdot)} \leq L(K, \nb).  	
	\end{equation}

	\end{Lemma}

\begin{proof}
	
\medskip

\noindent{\bf Step 1:}	
	
If \eqref{AS31} holds, then the estimates \eqref{AS17}, \eqref{AS27} give rise to 
\begin{align} 
\int_\tau^{\tau + 1} &\left( \| \vu \|^2_{W^{1,2}(\Omega; R^3)} + \| \vt^{\frac{\beta}{2}} \|^2_{W^{1,2}(\Omega)} 
+ \| \log(\vt) \|^2_{W^{1,2}(\Omega)}
+ \| \vB \|^2_{W^{1,2}(\Omega; R^3)} \right) \dt 
\br &\leq c_1(\nb) \mathcal{S}(r) \int_{\tau}^{\tau + 1} \intO{ \vr |\vu| } \dt + 
c_2(r,K, \nb).
 \label{AS33}	
	\end{align}

Next,  
\begin{align}
\int_\tau^{\tau +1} &\left( \| \vt \|^4_{L^4(\Omega)} + \intO{ |\Curl \vB \times \vB | } \right) \dt \br 
&\aleq \left(1 + \int_\tau^{\tau +1} \left( \| \vt^{\frac{\beta}{2}} \|^2_{W^{1,2}(\Omega)}  
+ \| \vB \|^2_{W^{1,2}(\Omega; R^3)} \right) \dt \right) \br 
&\leq c(K, \nb) \left( 1 + \int_{\tau}^{\tau + 1} \intO{ \vr |\vu| } \dt \right),
\nonumber
\end{align}
where we have used \eqref{AS33} with $r = 1$.
Consequently, inequality \eqref{AS29} reduces to
\begin{align}
	&\int_{\tau}^{\tau + 1} \intO{ \vr^{\frac{5}{3} + \alpha} } \dt \br &\leq c(K, \nb) \left[ 1 + 
	  \left( 1 + \sup_{t \in (\tau, \tau + 1)} \| \vr (t, \cdot) \|_{L^{\frac{5}{3}}(\Omega)} \right) 	
	\int_{\tau}^{\tau+1} \| \vu \|^2_{W^{1,2}(\Omega; R^3)}\dt + \sup_{t \in (\tau, \tau + 1)} 
	\intO{ \vr |\vu | } \right],\br &\quad \alpha = \frac{1}{15}.
	\nonumber
\end{align}
In addition, using once more \eqref{AS33} with $r=1$, together with the estimates \eqref{ww} and \eqref{AS30}, 
we may infer that 
\begin{align}
	&\int_{\tau}^{\tau + 1} \intO{ \vr^{\frac{5}{3} + \alpha} } \dt \br &\leq c(K, \nb) \left[ 1 + 
	\left( \sup_{t \in (\tau, \tau + 1)} \intO{ | F(\vr, \vt, \vu, \vB | } \right)^{\frac{3}{5}} 	
	\int_{\tau}^{\tau+1} \| \vu \|^2_{W^{1,2}(\Omega; R^3)} \dt \right. \br 
	&+ \left.  \left( \sup_{t \in (\tau, \tau + 1)} \intO{ | F(\vr, \vt, \vu, \vB | } \right)^{\frac{1}{2}} \right],\br &\quad \alpha = \frac{1}{15}.
	\label{AS34}
\end{align}
Here, we have used H\" older's inequality to estimate 
\begin{equation} \label{AS34b}
\intO{ \vr |\vu| } \leq \| \sqrt{\vr} \|_{L^2(\Omega)} \| \sqrt{\vr} \vu \|_{L^2(\Omega; R^3)} 
\leq c(\nb) \left[ 1 + \left( \intO{ |F(\vr, \vt, \vu, \vB| } \right)^{\frac{1}{2}} \right].
\end{equation}

\medskip 

\noindent{\bf Step 2:}

Similarly to \cite[Section 4.6]{FeiSwGw}, under the hypothesis \eqref{AS31}, we claim that 
the pointwise values of the functional $ \intO{ F }$ are controlled by its average: 
\begin{equation} \label{AS35}
\sup_{t \in (\tau, \tau+1)} \intO{ F(\vr, \vt, \vu, \vB) (t, \cdot) } 
\leq c(K, \nb) \left[ 1 + \int_{\tau}^{\tau+1}\intO{ F(\vr, \vt, \vu, \vB) (t, \cdot) }  \dt \right].
\end{equation}
Indeed there is $\xi \in (\tau, \tau +1)$ such that 
\[
\intO{ F(\vr, \vt, \vu, \vB) (\xi, \cdot) } \leq \int_{\tau}^{\tau+1}\intO{ F(\vr, \vt, \vu, \vB) (t, \cdot) }  \dt.
\] 
Next, using \eqref{AS27} with $r = 1$ and \eqref{AS34b}, we get 
\[
\frac{\D }{\dt} \intO{F(\vr, \vt, \vu, \vB) }\leq c(\nb) \left[ 1 + \intO{F(\vr, \vt, \vu, \vB) } \right].
\] 
Consequently,  
\begin{align}
\intO{F(\vr, \vt, \vu, \vB) (s, \cdot) } &\leq \intO{ F(\vr, \vt, \vu, \vB) (\xi, \cdot) } + 
\int_\xi^s  \frac{\D }{\dt} \intO{F(\vr, \vt, \vu, \vB) } \br 
&\leq c (\nb) \left[ 1 + \int_{\tau}^{\tau+1}\intO{ F(\vr, \vt, \vu, \vB) (t, \cdot) }  \dt \right] 
\label{AS35a}
\end{align}
whenever $\xi \leq s \leq \tau+1$.
Finally, by virtue of hypothesis \eqref{AS31} and \eqref{AS35a}, 
\begin{align}
\intO{ F(\vr, \vt, \vu, \vB) (\tau , \cdot) } &\leq K + \intO{ F(\vr, \vt, \vu, \vB) (\tau +1 , \cdot) } \br
&\leq K + c(\nb) \left[ 1 + \int_{\tau}^{\tau + 1} \intO{ F(\vr, \vt, \vu, \vB) } \dt \right].
\nonumber
\end{align}
Consequently, we may repeat the same arguments as 
in \eqref{AS35a} replacing $\xi$ by $\tau$, which yields \eqref{AS35}.

\medskip 

\noindent {\bf Step 3:}

We use H\" older's inequality and Sobolev embedding theorem $W^{1,2} \hookrightarrow L^6$
to bound the integral on the right--hand side of \eqref{AS33},
\[
\intO{ \vr |\vu| } \leq \| \sqrt{\vr} \|_{L^2(\Omega)} \| \sqrt{\vr} \|_{L^3(\Omega)} \| \vu \|_{L^6(\Omega; R^3)} 
\leq c \sqrt{m_0} \| \sqrt{\vr} \|_{L^3(\Omega)} \| \vu \|_{W^{1,2}(\Omega; R^3)}.
\]
Consequently, going back to \eqref{AS33} we obtain
\begin{align} \label{AS37}
	\int_{\tau}^{\tau+1} &\left( \| \vu \|^2_{W^{1,2}(\Omega; R^3)} + 
	\| \vt^{\frac{\beta}{2}} \|_{W^{1,2}(\Omega)}^2 + \| \vB \|_{W^{1,2}(\Omega; R^3)}^2 \right)  \dt \br 
& \leq c_1(\nb) \mathcal{S}(r) \int_{\tau}^{\tau + 1} \| \vr \|_{L^{\frac{3}{2}}(\Omega)}
	+ c_2 (r, K, \nb),
\end{align}
where possibly $c_2 (r, K, \nb) \to \infty$ as $r \to \infty$.
In particular, inequality \eqref{AS34} may be rewritten in the form
\begin{align}
	&\int_{\tau}^{\tau + 1} \intO{ \vr^{\frac{5}{3} + \alpha} } \dt \br &\leq c_1(K, \nb) \left[ 1 + 
	\left( \sup_{t \in (\tau, \tau + 1)} \intO{ | F(\vr, \vt, \vu, \vB | } \right)^{\frac{3}{5}} 	
	\mathcal{S}(r) \int_{\tau}^{\tau+1} \| \vr \|_{L^{\frac{3}{2}} (\Omega)} \dt \right. \br 
	&+ \left.  c_2(r, K, \nb )\left( \sup_{t \in (\tau, \tau + 1)} \intO{ | F(\vr, \vt, \vu, \vB | } \right)^{\frac{3}{5}} \right],\ \alpha = \frac{1}{15}.
	\label{AS38}
\end{align}

\medskip 

\noindent {\bf Step 4:}

Using \eqref{AS37} we derive a bound on the kinetic energy: 
\begin{align}
	\int_{\tau}^{\tau + 1} & \intO{ \vr |\vu|^2 } \dt \leq \sup_{t \in (\tau, \tau + 1)} \| \vr \|_{L^{\frac{3}{2}}(\Omega)} 
	\int_{\tau}^{\tau + 1} \| \vu \|^2_{L^6(\Omega; R^3)} \dt \br &\aleq \sup_{t \in (\tau, \tau + 1)} \| \vr \|_{L^{\frac{3}{2}}(\Omega)} 
	\int_{\tau}^{\tau + 1} \| \vu \|^2_{W^{1,2}(\Omega; R^3)} \dt \br 
	&\leq c_1 (r,K, \nb) \sup_{t \in (\tau, \tau + 1)} \| \vr \|_{L^{\frac{3}{2}}(\Omega)} + 
	c_2(K, \nb) \mathcal{S}(r) \sup_{t \in (\tau, \tau + 1)} \| \vr \|_{L^{\frac{3}{2}}(\Omega)}\int_{\tau}^{\tau + 1} \| \vr \|_{L^{\frac{3}{2}}(\Omega)} \dt. \nonumber
\end{align}
Next, by interpolation, 
\[
\| \vr \|_{L^{\frac{3}{2}}(\Omega)} \leq \| \vr \|_{L^{\frac{5}{3}}(\Omega)}^{\frac{5}{6}} \| \vr \|_{L^1(\Omega)}^{\frac{1}{6}};
\]	
whence 
\begin{align} 
	&\int_{\tau}^{\tau + 1}  \intO{ \vr |\vu|^2 }  \br 
	&\leq c_1 (r,K, \nb) \sup_{t \in (\tau, \tau + 1)} \| \vr \|_{L^{\frac{5}{3}}(\Omega)}^{\frac{5}{6}} 
	+ c_2(K, \nb) \mathcal{S}(r) \sup_{t \in (\tau, \tau + 1)} \| \vr \|_{L^{\frac{5}{3}}(\Omega)}^{\frac{5}{6}} \int_{\tau}^{\tau + 1} \| \vr \|_{L^{\frac{5}{3}}(\Omega)}^{\frac{5}{6}} \dt \br 
	&\leq  c_1(K, \nb) \left[ 1 + c_2(r,K, \nb) \left( \sup_{t \in (\tau, \tau+1)} \intO{ |F(\vr, \vt, \vu, \vB) |} \right)^{\frac{1}{2}} \right. \br  &+ \left. \left( \sup_{t \in (\tau, \tau+1)} \intO{ |F(\vr, \vt, \vu, \vB) |} \right)^{\frac{1}{2}} \mathcal{S}(r) \int_{\tau}^{\tau + 1}   \left(  \intO{ |F(\vr, \vt, \vu, \vB) |} \right)^{\frac{1}{2}} \dt  \right]  \br 
	&\leq  c_1(K, \nb) \left[ 1 + c_2(r,K, \nb) \left( \sup_{t \in (\tau, \tau+1)} \intO{ |F(\vr, \vt, \vu, \vB) |} \right)^{\frac{1}{2}} \right. \br & + \left. \mathcal{S}(r)  \sup_{t \in (\tau, \tau+1)} \intO{ |F(\vr, \vt, \vu, \vB) |}   \right] .                                           
	\label{AS39}
\end{align}	

\medskip 

\noindent {\bf Step 5:}

We estimate the pressure/internal energy term by interpolation between $L^1$ and $L^{\frac{5}{3} + \alpha}$,
\[
\int_{\tau}^{\tau + 1} \intO{ \vr^{\frac{5}{3}} } \dt \leq c(m_0) \left( \int_{\tau}^{\tau + 1} \intO{ \vr^{\frac{5}{3} + \alpha } }
\dt \right)^{\frac{10}{11}} \ \mbox{for}\ \alpha = \frac{1}{15}.
\]
Going back to \eqref{AS38} we conclude 
\begin{align}
	&\int_{\tau}^{\tau + 1} \intO{ \vr^{\frac{5}{3}} } \dt \br &\leq c_1(K, \nb) \left[ 1 + 
	\left( \sup_{t \in (\tau, \tau + 1)} \intO{ | F(\vr, \vt, \vu, \vB | } \right)^{\frac{3}{5}} 	
	\mathcal{S}(r) \int_{\tau}^{\tau+1} \| \vr \|_{L^{\frac{3}{2}} (\Omega)} \dt \right. \br 
	&+ \left.  c_2(r,K, \nb) \left( \sup_{t \in (\tau, \tau + 1)} \intO{ | F(\vr, \vt, \vu, \vB | } \right)^{\frac{3}{5}} \right]^{\frac{10}{11}} \br 
	&\quad \quad \mbox{\{ similarly to \eqref{AS39} \} } \br 
	&\leq c_1(K, \nb) \left[ 1 + \mathcal{S}(r)
	\left( \sup_{t \in (\tau, \tau + 1)} \intO{ | F(\vr, \vt, \vu, \vB | } \right)^{\frac{3}{5} + \frac{1}{2}} 	
 \right. \br 
	&+ \left.  c_2(r, K, \nb)  \left( \sup_{t \in (\tau, \tau + 1)} \intO{ | F(\vr, \vt, \vu, \vB | } \right)^{\frac{3}{5}} \right]^{\frac{10}{11}} \br 
&\leq c_1(K, \nb) \left[ 1 + \mathcal{S}(r)
\sup_{t \in (\tau, \tau + 1)} \intO{ | F(\vr, \vt, \vu, \vB | } \right. \br  	
&+ \left. c_2(r,K, \nb) \left( \sup_{t \in (\tau, \tau + 1)} \intO{ | F(\vr, \vt, \vu, \vB | } \right)^{\frac{6}{11}} \right].
	\label{AS40}
\end{align}

\medskip 

\noindent {\bf Step 6:}

Using inequality \eqref{AS35}, and summing up \eqref{AS37} - \eqref{AS40} to estimate the integral on the right--hand side we have 
\begin{align}
\sup_{t \in (\tau, \tau+1)} &\intO{ F(\vr, \vt, \vu, \vB) (t, \cdot) } 
\leq \sup_{t \in (\tau, \tau+1)} \intO{ | F(\vr, \vt, \vu, \vB) | (t, \cdot) } \br
&\leq c(K, \nb) \int_{\tau}^{\tau + 1} 
\left[ 1 + \intO{ \vr |\vu|^2 } + \| \vt^{ \frac{\beta}{2}} \|^2_{W^{1,2}(\Omega)} + 
	 \intO{\vr^{\frac{5}{3}} } + \| \vB \|^2_{W^{1,2}(\Omega; R^3)} \right] \dt	\br 
	 &\leq c_1(K, \nb) \left[ 1 + c_2(r,K, \nb) \left( \sup_{t \in (\tau, \tau + 1)} \intO{ | F(\vr, \vt, \vu, \vB | } \right)^{\frac{6}{11}} \right. \br 
	 &+ \left. \mathcal{S}(r)\sup_{t \in (\tau, \tau + 1)} \intO{ | F(\vr, \vt, \vu, \vB | }  \right].
	\label{AS41}
	\end{align}
Thus choosing $r = r(K, \nb)$ so large that 
\[
c_1(K, \nb) \mathcal{S}(r) < 1 
\]
in \eqref{AS41} we obtain the desired conclusion \eqref{AS32}.

	\end{proof}

Having proved Lemma \ref{LAS1} we are ready to finish the proof of Theorem \ref{TAS1}. Suppose that 
\[ 
\limsup_{t \to T-} \intO{ E(\vr, \vt, \vu, \vB )(t, \cdot) } = \mathcal{E}_T < \infty. 
\]
In view of \eqref{AS30}, 
\[
\limsup_{t \to T-} \intO{ F(\vr, \vt, \vu, \vB )(t, \cdot) } \leq \mathcal{F}_T (\mathcal{E}_T, \nb) < \infty.
\]
Fix $K = 1$ in Lemma \ref{LAS1}. First we claim there is $\tau \geq T$, where $\tau$ depends only on 
$\mathcal{F}_T$ and $\nb$ such that 
\begin{equation} \label{AS42}
\intO{ F(\vr, \vt, \vu, \vB )(\tau, \cdot) } \leq L(1, \nb), 
\end{equation}	
where $L(1, \nb)$ is the bound in \eqref{AS32}. More precisely,  
\[
\tau - T \leq [ \mathcal{F}_T + |\Omega| \inf F ] + 1 
\]
since otherwise, in accordance with Lemma \ref{LAS1}, we would have 
\[
\intO{ F(\vr, \vt, \vu, \vB ) (\tau, \cdot) } < \intO{ \inf F }.
\]

Next, we claim the implication 
\[
\intO{ F(\vr, \vt, \vu, \vB) (\tau, \cdot) } \leq L(1, \nb) \ \Rightarrow 
\ \intO{ F(\vr, \vt, \vu, \vB) (\tau + 1, \cdot) } \leq L(1, \nb). 
\]
Indeed either 
\[
\intO{ F(\vr, \vt, \vu, \vc{B} ) (\tau, \cdot) } - 
\intO{ F(\vr, \vt, \vu, \vc{B} ) (\tau + 1, \cdot) } \leq 1
\]
and the desired conclusion follows from Lemma \ref{LAS1}, or 
\[
\intO{ F(\vr, \vt, \vu, \vc{B} ) (\tau + 1, \cdot) } < \intO{ F(\vr, \vt, \vu, \vc{B} ) (\tau, \cdot) } - 1 
< L(1, \nb).
\]

Thus we have a sequence 
\[
\tau + n,\ n = 0,1,\dots,\ 
\intO{ F(\vr, \vt, \vu, \vc{B} ) (\tau + n, \cdot) } \leq L(1, \nb)
\]
and the desired uniform bound 
\[
\sup_{t \geq \tau} \intO{ F(\vr, \vt, \vu, \vc{B} ) (t, \cdot) } \leq \mathcal{F}_\infty (\nb) 
\]
follows from \eqref{AS27} (with $r=1$) by the standard Gronwall argument. 
Clearly, in view of \eqref{AS30}, this yields a similar bound for $\intO{ E(\vr, \vt, \vu, \vB)}$ as 
claimed in Theorem \ref{TAS1}.

\section{Asymptotic compactness}
\label{AC}

The proof of asymptotic compactness basically mimicks the proof of sequential stability of global in time weak solutions 
with the necessary modification to ensure compactness of the density. 

We consider a family of data locally \emph{precompact} in the appropriate function spaces. Specifically, 
\begin{equation} \label{AC1}
0 < \underline{m}_0 \leq m_{0,n} \leq \Ov{m}_0,\ |\bfomega_n| \leq \omega,
	\end{equation}
\begin{equation} \label{AC2}
	\| G_n \|_{W^{1,\infty}(R \times \Omega)} \leq \Ov{G},\ 
	(G_n)_{n \geq 1} \ \mbox{precompact in}\ W^{1,\infty}([-N,N] \times \Omega) \ \mbox{for any}\ N > 0.
	\end{equation}
Similarly, 
\begin{align}
0 < \underline{\vt} \leq \vt_{B,n},\ 
&\sup_{t \in R} \left( \| \vt_{B,n} (t, \cdot) \|_{W^{2, \infty}(\Gamma^\vt_D)} + 
	\| \partial_t \vt_{B,n} (t, \cdot) \|_{W^{1, \infty}(\Gamma^\vt_D)} \right) \leq \Ov{\vt}, \br 
&(\vt_{B,n} )_{n \geq 1} \ \mbox{precompact in}\ C([-N, N]; C^2(\Gamma^\vt_D)) \ \mbox{for any}\ N > 0,\br 
&(\partial_t \vt_{B,n} )_{n \geq 1} \ \mbox{precompact in}\ C([-N, N]; C^1(\Gamma^\vt_D)) \ \mbox{for any}\ N > 0;	
	\label{AC3}
	\end{align}
\begin{align} 
	\sup_{t \in R} &\left( \| \vc{b}_{\tau,n} (t, \cdot) \|_{W^{2, \infty}(\Gamma^\vB_D; R^3)} + 
	\| \partial_t \vc{b}_{\tau,n}  (t, \cdot) \|_{W^{1, \infty}(\Gamma^\vB_D; R^3)} \right) \leq \Ov{B}_\tau, \br 
	&(\vc{b}_{\tau,n} )_{n \geq 1} \ \mbox{precompact in}\ C([-N, N]; C^2(\Gamma^\vB_D; R^3)) \ \mbox{for any}\ N > 0,\br 
	&(\partial_t \vc{b}_{\tau,n} )_{n \geq 1} \ \mbox{precompact in}\ C([-N, N]; C^1(\Gamma^\vB_D; R^3)) \ \mbox{for any}\ N > 0;	
	\label{AC4}
\end{align}
and
\begin{align} 
	\sup_{t \in R} &\left( \| {b}_{\nu,n} (t, \cdot) \|_{W^{2, \infty}(\Gamma^\vB_N)} + 
	\| \partial_t \vc{b}_{\nu,n}  (t, \cdot) \|_{W^{1, \infty}(\Gamma^\vB_N)} \right) \leq \Ov{B}_\nu, \br 
	&({b}_{\nu,n} )_{n \geq 1} \ \mbox{precompact in}\ C([-N, N]; C^2(\Gamma^\vB_N)) \ \mbox{for any}\ N > 0,\br 
	&(\partial_t {b}_{\nu,n} )_{n \geq 1} \ \mbox{precompact in}\ C([-N, N]; C^1(\Gamma^\vB_N)) \ \mbox{for any}\ N > 0.	
	\label{AC5}
\end{align}

\begin{mdframed}[style=MyFrame]
	
	\begin{Theorem} \label{TAC1} {\bf [Asymptotic compactness]}
		
		Let the hypotheses of Theorem \ref{TAS1} hold. Suppose that $(\vrn, \vtn, \vun, \vB_n )_{n \geq 1}$ is a 
		sequence of weak solutions to the compressible MHD system on $(T_n, \infty)$, $T_n \geq - \infty$, $T_n \to - \infty$, 
		with the data 
		\[
		(m_{0,n}, \omega_n, G_n, \vt_{B,n}, \vc{b}_{\tau,n}, b_{n,n} )_{n\geq 1}
		\] specified in 
		\eqref{AC1}--\eqref{AC5}, 
		\[
		\intO{ \vrn } = m_{0,n},\ \intO{ \vB_n \cdot \vc{h} } = 0 
		\ \mbox{for all}\ \vc{h} \in \mathcal{H}(\Omega; R^3).
		\]
		In addition, suppose 
		\[
		\limsup_{\tau \to T_n+ } \intO{ E(\vrn, \vtn, \vun, \vB_n)(\tau, \cdot) } < \mathcal{E}_0 < \infty 
		\]
		uniformly for all $n=1,2,\dots$.
		
		Then there is a subsequence (not relabelled here) such that 
		\begin{align} 
		\vrn &\to \vr \ \mbox{in}\ C_{\rm weak}([-N,N]; L^{\frac{5}{3}}(\Omega)) 
		\ \mbox{and in}\ C([-N,N]; L^1(\Omega)), \br 
		\vtn &\to \vt \ \mbox{in}\ L^q ([-N,N]; L^{4}(\Omega)) \ \mbox{and weakly in}
		\ L^2(-N,N; W^{1,2}(\Omega)),\ 1 \leq q < \infty, \br 
		\vun &\to \vu \ \mbox{weakly in}\  L^2(-N,N; W^{1,2}(\Omega; R^3)), \br
		\vB_n &\to \vB \ \mbox{in}\ L^2(-N,N; L^2(\Omega; R^3)) \ \mbox{and weakly in}\ 
		L^2(-N,N; W^{1,2}(\Omega; R^3)),  	
			\label{AC6}
			\end{align}
and 
\begin{align}
	m_{0,n} &\to m_0 > 0,\ \bfomega_n \to \bfomega, \br 
G_n &\to G 	\ \mbox{in} \ W^{1,\infty}([-N,N] \times \Omega), \br 
\vt_{n,B}&\to \vt_B \ \mbox{in}\  C([-N, N]; C^2(\Gamma^\vt_D)) ,\ \inf_{R \times \Omega} \vtB > 0, \br 
\vc{b}_{\tau, n} &\to \vc{b}_\tau \ \mbox{in}\ C([-N, N]; C^2(\Gamma^\vB_D; R^3)), \br
b_{\nu,n} &\to b_\nu \ \mbox{in}\  C([-N, N]; C^2(\Gamma^\vB_N))
	\label{AC7} 
\end{align}		
for any $N > 0$, where $(\vr, \vt, \vu, \vB)$ is a weak solution of the compressible MHD system with the data 
$(m_0, \omega, G, \vtB, \vc{b}_\tau, b_\nu)$ defined in $R \times \Omega$ (entire solution) satisfying 
\begin{equation} \label{AC8}
	\intO{ E(\vr, \vt, \vu, \vB) (\tau, \cdot) } \leq \mathcal{E}_\infty 
	\ \mbox{for all}\ \tau \in R.
	\end{equation}	
		\end{Theorem}
	
	\end{mdframed}

As already pointed out, the proof of Theorem \ref{TAC1} is based on ideas similar to the proof of sequential stability 
of weak solutions on a \emph{compact} time interval $(-N,N)$. Here the necessary energy bounds are provided by 
Theorem \ref{TAS1}; whence the only missing piece of information is the strong convergence of the densities
at the ``initial'' time $t = -N$. Indeed compactness 
of the sequence $(\vrn )_{n \geq 1}$ in the existence proof proposed by Lions \cite{LI4} or \cite{EF70} is conditioned by compactness of the ``initial'' density distributions. In the present setting, however, 
the ``initial'' state is {\it a priori} not given. 

Fortunately, this problem can be solved by careful analysis 
of propagation of the density oscillations in time elaborated in \cite{FP15} and further developed in 
\cite{FanFeiHof} in the context of inhomogeneous velocity boundary conditions. Adaptation to the present setting 
is straightforward as soon as we observe compactness of the Lorentz force in the momentum equation, specifically 
\[
\Curl \vB_n \times \vB_n \to 
\Curl \vB \times \vB \ \mbox{weakly in}\ L^{\frac{3}{2}}((-N,N) \times \Omega; R^3). 
\]
However, given the available energy bounds, this can be shown by the standard Aubin--Lions argument. The remaining 
compactness properties necessary to establish the convergence in \eqref{AC6} are well understood 
and elaborated in detail in \cite{DUFE2}.  

\section{Concluding discussion, applications}
\label{AP}

Suppose that the gravitational potential $G$ as well as the boundary data $\vtB$, $\vc{b}_\tau$, $b_\nu$ are independent of the time. Under the hypotheses of Theorem \ref{TAS1}, the set 
\begin{align}
	\mathcal{A} = \Big\{ (\vr, \vt, \vu, \vB) \ \Big| & 
	(\vr, \vt, \vu, \vB) - \mbox{solution of the compressible (MHD) system for}\ t \in (-\infty, \infty) \br
	& \intO{ E(\vr, \vt, \vu, \vB ) } < \mathcal{E}_\infty \Big\}
	\nonumber
\end{align}
is non--empty, time--shift invariant, meaning 
\[
(\vr, \vt, \vu, \vB) \in \mathcal{A} \ \Rightarrow \ (\vr, \vt, \vu, \vB) (\cdot + T) \in \mathcal{A},  
\ T \in R,
\]
and compact with respect to the topologies specified in \eqref{AC6}.

\subsection{Convergence to equilibria}

A particular situation occurs when $\mathcal{A}$ is a singleton containing the unique stationary solution 
of the problem. As a direct consequence of Theorem \ref{TAC1} we obtain unconditional convergence to equilibrium 
for \emph{any} weak solution of the compressible MHD system. 

As an example, consider the situation:
\begin{itemize}
	\item $\vtB > 0$ is a positive constant on $\Gamma^\vt_D \ne \emptyset$;
	\item 
	\[
	\vc{b}_{\tau} = \bB \times \vc{n}|_{\Gamma^\vB_D},\ 
	b_\nu = \bB \cdot \vc{n}|_{\Gamma^\vB_N}, 
	\]
	where 
	\[
	\Curl \bB = \Div \bB = 0 \ \mbox{in} \ \Omega.
	\]
\end{itemize}
Accordingly, the ballistic energy inequality \eqref{AS16} yields
\begin{align}
	&\frac{\D }{\dt}	\intO{ \left( \frac{1}{2} \vr |\vu|^2 + \vr e(\vr, \vt) + \frac{1}{2} |\vB |^2 - \vtB \vr s(\vr, \vt) -  \vr M \right) }  \br
	\quad&
	+ \intO{ 	\frac{\vtB}{\vt} \left( \mathbb{S}(\vt, \Grad \vu) : \Grad \vu - \frac{\vc{q}(\vt, \Grad \vt) \cdot \Grad \vt}{\vt} + \zeta (\vt) | \Curl \vB |^2 \right) }  
	 \leq 0.
	\label{AP1}
\end{align}	

Integrating \eqref{AP1} in time we deduce that any solution belonging to $\mathcal{A}$ necessarily satisfies 
\[
\vu \equiv 0, \ \vt \equiv \vtB ,\ \partial_t \vr = 0,\ \Div \vB = \Curl \vB = 0.
\]
Moreover, it turns out that $\vB$ is uniquely determined by the boundary conditions and the hypothesis 
\[
\intO{ \vc{B} \cdot \vc{h} } = 0 \ \mbox{for all}\ \vc{h} \in \mathcal{H}(\Omega).
\]

Finally, $\vr = \vr(x)$ solves the static problem 
\[
\Grad p(\vr, \vtB) = \vr \Grad M  
\]
that admits a unique solution determined by the mass constraint 
\[
\intO{ \vr } = m_0
\]
as long as $P'(0) > 0$ in \eqref{w2}, see \cite{FP20}.

\subsection{Systems with unbounded energy}

As mentioned in the introduction, the result claimed in Theorem \ref{TAS1} cannot hold if the fluid system is thermally isolated, meaning 
\begin{equation} \label{AP2}
\Gamma^\vt_D = \emptyset.
\end{equation}
Indeed it follows from the entropy inequality \eqref{w25} that the total entropy 
\[
\tau \mapsto \intO{ \vr s(\vr, \vt) (\tau, \cdot) } 
\]
is a non--decreasing function, which precludes the existence of a bounded absorbing set 
depending solely on the data specified in \eqref{AS1}.

What is more, we show that \eqref{AP2} may give rise to trajectories with unbounded energy. 
Indeed suppose that there is a global--in--time solution $(\vr, \vt, \vu, \vB)$ and a sequence of times 
$\tau_n \to \infty$ such that 
\begin{equation} \label{AP3}
	\sup_{n \geq 1} \intO{ E(\vr, \vt, \vu, \vB)(\tau_n, \cdot) } < \infty.
	\end{equation}
It follows from the constitutive restrictions \eqref{ww}  	
that boundedness of total energy implies boundedness of the total entropy $\intO{ \vr s(\vr, \vt) (\tau_n, \cdot)}$. However, by virtue of the no--flux boundary conditions induced by \eqref{AP2}, the total entropy 
is a non--decreasing function of the time. Consequently, we may infer that 
\begin{equation} \label{AP4}
	\int_T^\infty \intO{ \frac{1}{\vt} \left( \mathbb{S}(\vt, \Grad \vu) : \Grad \vu - \frac{\vc{q}(\vt, \Grad \vt) \cdot \Grad \vt}{\vt} + \zeta(\vt) |\Curl \vB |^2 \right) } \dt < \infty.
		\end{equation} 
	
In particular, it follows from \eqref{AP4} that 
\begin{equation} \label{AP5}
	\int_{\tau_n}^{\tau_n + 1} \intO{ \frac{1}{\vt} \left( \mathbb{S}(\vt, \Grad \vu) : \Grad \vu + 
\zeta(\vt) |\Curl \vB |^2 \right) } \dt \to 0 \ \mbox{as}\ n \to \infty.
\end{equation} 				
Consequently, introducing the time shifts 
\[
\vun = \vu(\cdot + \tau_n),\ \vB_n = \vB( \cdot + \tau_n ),
\]
we deduce 
\[
\vu_n \to 0 \ \mbox{in}\ L^2(0,1; W^{1,2}(\Omega; R^3),\ 
\vB_n \to \widehat{\vB} \ \mbox{in}\ L^2(0,1; L^2(\Omega; R^3)) \ \mbox{and weakly in}\ 
L^2(0,1; W^{1,2}(\Omega; R^3)).
\]
As the limits satisfy the induction equation \eqref{p3}, we conclude 
\begin{equation} \label{AP6}
	\widehat{\vB}  = \widehat{\vB} (x),\ \Div \widehat{\vB} = 0, \ \Curl \widehat{\vB} = 0 \ \mbox{in}\ \Omega,\ 
	\widehat{\vB} \times \vc{n}|_{\Gamma^\vB_D} = \vc{b}_{\tau},\ 
	\widehat{\vB} \cdot \vc{n}|_{\Gamma^\vB_N} = b_\nu.
\end{equation}	
Thus the boundary data $b_\nu$, $\vc{b}_\tau$ must be stationary in the sense of Definition \ref{Dew1}.
Otherwise, the conclusion contradicts \eqref{AP3}.

Using the necessary condition for stationarity \eqref{p20a} we obtain the following result.

\begin{mdframed}[style=MyFrame]

\begin{Theorem} \label{TAP1} {\bf [Systems with unbounded energy]}
	
Suppose 
\[
\Gamma^\vu_D \ne \emptyset,\ \Gamma^\vt_D = \emptyset,
\]
and $b_\nu$, $\vc{b}_\tau$ independent of $t$, 
\[
\Gamma^\vB_D \ne \emptyset,\ \mbox{where}\ \vc{b}_\tau \cdot \vc{n} = 0,\ 
{\rm div}_\tau \vc{b}_\tau \not\equiv 0.
\]

Then 
\[
\intO{ E(\vr, \vt, \vu, \vB)(\tau, \cdot) } \to \infty \ \mbox{as}\ \tau \to \infty
\]
for any global--in--time weak solution to the compressible MHD system.

\end{Theorem}	

\end{mdframed}

\subsection{General boundary fluxes}
\label{GF}

The results can be extended to general boundary fluxes. Prominent examples are 
the Navier--slip boundary condition for the velocity, 
\begin{equation} \label{AP7}
\vu \cdot \vc{n}|_{\Gamma^\vu_N} = 0, 
\ \left[ \mathbb{S}(\vt, \Grad \vu) \cdot \vc{n} + d \vu \right] \times \vc{n}|_{\Gamma^\vu_N} = 0, \ d > 0,	
	\end{equation} 
see e.g. Bul{\' \i}{\v c}ek,  M{\' a}lek, and Rajagopal \cite{BMR}, 
and the radiative heat flux condition 
\begin{equation} \label{AP8}
\vc{q} (\vt, \Grad \vt) \cdot \vc{n} - d |\vt - \vt_0|^k (\vt - \vt_0)|_{\Gamma^\vt_N} = 0 	
	\end{equation}
used in some models in astrophysics, cf. Weiss and Proctor \cite{Weiss}. 

If $d > 0$, then the conditions $\Gamma^\vu_D \ne \emptyset$, $\Gamma^\vt_D \ne \emptyset$ can be dropped in 
Theorem \ref{TAS1}. Hypothesis $\Gamma^\vB_D \subset \Gamma^\vu_D$, however, is still necessary. 
We leave the details to the interested reader.

\subsection{``Do nothing'' boundary conditions}
\label{DN} 

Another interesting type of boundary conditions arises under a simple but physically grounded assumption 
that the Maxwell equations, here represented by the induction equation \eqref{p3}, are satisfied both in 
$\Omega$ and in the adjacent component $\Omega_{\rm ext}$ 
with possibly different (discontinuous) transport coefficients. The boundary conditions on $\Gamma \subset \partial \Omega$ separating $\Omega$ and $\Omega_{\rm ext}$ are then nothing other than the Rankine--Hugeniot compatibility conditions to ensure the satisfaction of the induction equation in the weak sense. 

More specifically, we write the Maxwell's equations yielding \eqref{p3} in a more precise way: 
\begin{align} \label{MAX}
\partial_t \vc{B} + \Curl \vc{E} &= 0, \br 
\Div \vc{B} &= 0, \br
\vc{J} &= \sigma \left( \vc{E} + \vu \times \vc{B} \right), \br
\vc{J} &= \Curl \vc{H},\ \vc{B} = \Ov{\mu} \vc{H},  	
	\end{align}  
where $\vc{E}$ denotes the electric field, $\vc{J}$ the electric current density, $\vc{H}$ the magnetic field intensity, $\sigma$ is the electric conductivity and $\Ov{\mu}$ the magnetic permitivity of the material. 
Both $\sigma$ and $\Ov{\mu}$ may experience jumps over $\Gamma$.
Accordingly, the Lorentz force in the momentum equation reads 
\[
\vc{J} \times \vc{B} = \Curl \vc{H} \times \vc{B}.
\]

The associated 
Rankine-Hugeniot conditions on $\Gamma$ are:
\begin{align} \label{RH}
	\vc{B} \cdot \vc{n} &= \vc{B}_{\rm ext} \cdot \vc{n}, \br
	\vc{H} \times \vc{n} &= \vc{H}_{\rm ext} \times \vc{n}, \br 
	\vc{E} \times \vc{n} &= \vc{E}_{\rm ext} \times \vc{n}.
\end{align}
If the adjacent component $\Omega_{\rm ext}$ happens to be a vacuum, we have $\sigma_{\rm ext} = 0$, and 
\eqref{MAX}, \eqref{RH} give rise to a non--local condition for $\vc{B}$ on $\Gamma$ given by a Neumann--Dirichlet 
type map, see Ladyzhenskaya and Solonnikov \cite{LadSol}, Sakhaev and Solonnikov \cite{SakSol}, Solonnikov \cite{Solo}.  

At the level of the energy estimates yielding the existence of a bounded absorbing set, the above situation 
does not create any additional difficulties. Indeed, in view of \eqref{RH}, the function 
\[
\Ov{\mu} ( \vc{B} - \vc{B}_B) =  
\vc{H} - \vc{H}_B,\ \Ov{\mu} = \Ov{\mu}_{\rm int} \mathds{1}_\Omega + \Ov{\mu}_{\rm ext} \mathds{1}_{\Omega_{\rm ext}}
\] 
is still a legitimate test function for the induction equation \eqref{p3} yielding
\begin{align} \label{BAL}
	\frac{\D }{\dt} \int_{\Omega \cup \Omega_{\rm ext}} \Ov{\mu} \left( \frac{1}{2} |\vc{B}|^2 - \bB \cdot \vc{B} \right) \dx &=  \int_{\Omega \cup \Omega_{\rm ext}} (\vu \times \vc{B})  \cdot \Curl (\vc{H} - \vc{H}_B) \dx \br &+ \int_{\Omega \cup \Omega_{\rm ext}} \frac{1}{\sigma} \Curl \vc{H} \cdot \Curl (\vc{H} - \vc{H}_B) \dx.
	\end{align}
If we set the velocity $\vu_{\rm ext} = 0$, all energy bounds derived in the proof of Theorem \ref{TAS1} remain valid; whence the existence of a bounded absorbing set follows from Theorem \ref{TAS1}

The proof of asymptotic compactness claimed in Theorem \ref{TAC1} may be more delicate. Still we conjecture that the results of Bauer, Pauly, and Schomburg \cite{BaPaSc} may be used to show the asymptotic compactness 
for the sequence of magnetic fields to extend Theorem \ref{TAC1} to the case of ``do nothing'' boundary conditions. 

\bigskip

\centerline{\bf Acknowledgement}

The work was done during the stay of E.Feireisl at the Faculty of Arts and Science of Kyushu University, Fukuoka,
Japan, 
which support is gladly acknowledged.



\end{document}